\documentclass[12pt]{amsart}
\usepackage{a4wide}
\usepackage[utf8]{inputenc}
\usepackage{amsmath}
\usepackage{amssymb}
\usepackage{amsfonts}
\usepackage{graphicx}
\usepackage{enumerate}
\usepackage{color}
\usepackage{mathtools}
\usepackage{mathrsfs}
\usepackage[colorlinks,linkcolor=blue,anchorcolor=blue,citecolor=blue]{hyperref}

\newtheorem{theorem}{Theorem}[section]
\newtheorem{proposition}[theorem]{Proposition}
\newtheorem{lemma}[theorem]{Lemma}
\newtheorem{definition}[theorem]{Definition}
\newtheorem{corollary}[theorem]{Corollary}
\theoremstyle{definition}
\newtheorem{example}[theorem]{Example}
\theoremstyle{remark}
\newtheorem{remark}[theorem]{Remark}

\numberwithin{equation}{section}

\newcommand{\R}{\mathbb{R}}
\newcommand{\Q}{\mathbb{Q}}

\newcommand{\N}{\mathbb{N}}

\newcommand{\f}{\infty}

\newcommand{\Ga}{\Gamma}
\renewcommand{\t}{\boldsymbol{t}}
\renewcommand{\i}{\mathbf{i}}
\renewcommand{\j}{\mathbf{j}}
\newcommand{\I}{\mathcal{I}}

\title[Self-similarity of unions]{Self-similarity of unions of self-similar sets and their translations}

\author{Zhiqiang Wang}
\address[Zhiqiang Wang]{College of Mathematics and Statistics, Center of Mathematics, Chongqing University, Chongqing 401331, People's Republic of China \;\&\; Department of Mathematics, University of British Columbia, Vancouver, British Columbia, V6T 1Z2, Canada}
\email{zhiqiangwzy@163.com,~zqwangmath@cqu.edu.cn}

\date{\today}

\begin{document}
\begin{abstract}
In this paper, we explore the self-similarity of unions of self-similar sets and their translations.
For $N \in \N$ and $0< \beta < 1/(N+1)$, let $\Gamma$ be the self-similar set generated by the IFS
\[ \Big\{ \phi_i(x)=\beta x + i \frac{1-\beta}{N}: i=0,1,\ldots, N \Big\}. \]
We provide a complete characterization of translation vectors $\boldsymbol{t} =(t_0,t_1, \ldots, t_m) \in \R^{m+1}$ with $0=t_0 < t_1 < \cdots < t_m$ for which the union $\bigcup_{j=0}^m (\Gamma+t_j)$ is a self-similar set, by determining the existence of cycles in
associated directed graphs.
This extends the result of Kong \emph{et al.} \cite{KLWYZ-2024}.
Additionally, we present two types of self-similar sets for which the union with their translations cannot be self-similar.
\end{abstract}

\maketitle

\section{Introduction}

Self-similarity, which refers to a set consisting of smaller similar copies of itself, is a central concept in fractal geometry.
A self-similar \emph{iterated function system} (IFS) on $\R^d$ means a finite set of contractive similitudes, i.e., $\mathcal{F}=\{ f_i(x) = \lambda_i O_i x + a_i \}_{i=1}^{\ell}$ where $0< |\lambda_i| < 1$, $O_i$ is an orthogonal real matrix, and $a_i \in \R^d$.
According to Hutchinson \cite{Hutchinson-1981}, any self-similar IFS $\mathcal{F}$ on $\R^d$ has a unique \emph{attractor}, i.e., a unique non-empty compact subset $K \subset \R^d$ satisfying
\begin{equation}\label{eq:attractor}
  K = \bigcup_{i=1}^\ell f_i (K).
\end{equation}
If the union on the right-hand side in (\ref{eq:attractor}) is pairwise disjoint, then we say that the IFS $\mathcal{F}$ satisfies the \emph{strong separation condition} (SSC).
If there exists a non-empty open subset $U \subset \R^d$ such that $\bigcup_{i=1}^\ell f_i(U) \subset U$ with the union on the left-hand side pairwise disjoint, then we say that the IFS $\mathcal{F}$ satisfies the \emph{open set condition} (OSC).
A non-empty compact subset $K \subset \R$ is called a \emph{self-similar set} if it is the attractor of a self-similar IFS.
A self-similar set $K$ is said to satisfy the SSC (resp. OSC) if it is the attractor of a self-similar IFS that satisfies the SSC (resp. OSC).
To avoid triviality, we assume throughout the paper that any self-similar set under consideration is non-singleton.

A natural question is to determine whether a given set is self-similar, which is generally difficult. For example, it is easy to show that a closed ball in $\R^d$ is not self-similar, and that any convex polytope in $\R^d$ is a self-similar set. However, for general convex compact subsets with non-empty interior, no characterization of self-similarity is currently available. Nonetheless, there is some literature that clarifies the self-similarity of certain special sets, such as intersections of Cantor sets \cite{Deng-He-Wen-2008,Kong-Li-Dekking-2010,Zou-Lu-Li-2008,Li-Yao-Zhang-2011,Yao-Li-2012,
Pedersen-Phillips-2014,Baker-Kong-2017}, subsets of the middle-third Cantor set \cite{Feng-Rao-Wang-2015}, graphs of continuous functions \cite{Bandt-Kravchenko-2011,Moreira-Xi-Zhang-2025}, concave quadrangles \cite{Xu-Xi-Jiang-2024} and bounded regions in $\R^2$ with $C^2$ boundary \cite{Wang-Jiang-Xi-2025}.
The union of finitely many intervals that is a self-similar set satisfying the OSC was characterized in \cite{Feng-Hua-Ji-2007,Wen-Zhao-2017}.

In this paper, we study the self-similarity of unions of self-similar sets and their translations in $\R$.
More precisely, for a self-similar set $K$ in $\R$ and a translation vector $\t=(t_0, t_1, \ldots, t_m) \in \R^{m+1}$, where $m \in \N$, we aim to determine whether the union
\[ K_{\t} := \bigcup_{j=0}^m(K+t_j)\] is a self-similar set.
Note that any translation of a self-similar set is again a self-similar set.
Without loss of generality, we assume throughout the paper that the translation vector $\t=(t_0, t_1, \ldots, t_m)\in\R^{m+1}$ always satisfies $0=t_0<t_1<\cdots<t_m$.

\subsection{Homogeneous symmetric self-similar sets}

We first focus on homogeneous symmetric self-similar sets.
Fix a positive integer $N \in \N$ and a real number $0< \beta < 1/(N+1)$. Let $\Ga=\Ga_{\beta, \{ 0,1,\ldots, N\}}$ be the attractor of the self-similar IFS
\begin{equation}\label{eq:IFS-Ga}
   \Big\{ \phi_i(x)=\beta x + i \frac{1-\beta}{N}: i=0,1,\ldots, N \Big\}.
\end{equation}
The set $\Ga$ can be written as
\begin{equation*}
 \Ga=\bigg\{\frac{1-\beta}{N} \sum_{k=1}^\f i_k \beta^{k-1}: i_k\in\{0,1,\ldots, N\}~\forall k \in \N\bigg\}.
\end{equation*}
Note that $\Ga$ is symmetric, i.e., $\Ga=1-\Ga$.
For a translation vector $\t=(t_0, t_1, \ldots, t_m) \in \R^{m+1}$ with $0=t_0<t_1<\cdots<t_m$, define
\begin{equation}\label{eq:Ga-t}
  \Ga_{\t}:=\bigcup_{j=0}^m(\Ga+t_j).
\end{equation}

In \cite{Deng-Liu-2011}, Deng and Liu first provided a complete characterization of the self-similarity of $\Ga \cup (\Ga + t)$ for $N=1$ and $\beta = 1/q$ where $q \ge 3$ is an integer.
In the general setting, Kong \emph{et al.} \cite{KLWYZ-2024} showed that for any $m \in \N$, there are infinite translation vectors $\t \in \R^{m+1}$ such that the union $\Ga_{\t}$ is a self-similar set.
Furthermore, for $0< \beta < 1/(2N+1)$, they gave a complete characterization of the self-similarity of the union $\Ga_{\t}$ by determining whether an associated directed graph contains a cycle.
We remark that the result of Kong \emph{et al.} \cite{KLWYZ-2024} fails to include the case $N=1$ and $\beta=1/3$.
In this paper, we extend the result of Kong \emph{et al.} \cite{KLWYZ-2024} to all cases $0< \beta<1/(N+1)$.

%For a translation vector $\t=(t_0, t_1,\ldots, t_m)\in\R^{m+1}$ we are interested in whether the union
%\[\Ga_{\t}:=\bigcup_{j=0}^m(\Ga+t_j)\] is a self-similar set.
%Without loss of generality, we assume throughout the paper that the translation vector $\t=(t_0, t_1, \ldots, t_m)\in\R^{m+1}$ always satisfies $0=t_0<t_1<\cdots<t_m$.

Let $\Omega:=\{0,1,\ldots,N\}$, and for $\i= i_1 i_2 \ldots i_n \in \Omega^n$ define $\phi_{\i}:=\phi_{i_1} \circ \phi_{i_2} \circ \cdots \circ \phi_{i_n}$.
For $n \in \N$, let $E_n := \big\{ \phi_{\i}(0): \i \in \Omega^n \big\}$, and $T_n := \beta^{-n} E_n$. Define \[ T := \bigcup_{n=1}^\f T_n. \]
Note that the sequence $\{T_n\}_{n=1}^\f$ is increasing with respect to the set inclusion.
For $\t =(t_0, t_1, \ldots, t_m) \in T^{m+1}$ with $0=t_0<t_1<\cdots<t_m$, we can find the smallest integer $\tau_{\t} \in \N$ such that $ t_1, \ldots, t_m \in T_{\tau_{\t}}$.
For $n\ge \tau_{\t}$, define
\begin{align*}
  \Omega_{\t}^n & := \big\{ \i \in \Omega^n:  \phi_{\i}(t_j) \in E_n \;\; \forall 1\le j \le m  \big\}, \\
  \widehat{\Omega}_{\t}^n & := \big\{ \i \in \Omega^n:  \phi_{\i}(-t_j) \in E_n \;\; \forall 1\le j \le m  \big\}.
\end{align*}

\begin{definition}\label{def:admissible-translation}
  A vector $\t=(t_0, t_1 , \ldots , t_m) \in \R^{m+1}$ with $0=t_0 < t_1 < \cdots < t_m$ is an \emph{admissible translation vector} if $\t \in T^{m+1}$, and there exist two finite sets $\I_1\subset \bigcup_{n=\tau_{\t}}^\f \Omega_{\t}^n$ and $\I_2 \subset \bigcup_{n=\tau_{\t}}^\f \widehat{\Omega}_{\t}^n$ such that
  \[ \bigcup_{\i\in \I_1}\phi_{\i}(\Ga_{\t})\cup\bigcup_{\i\in\I_2}\phi_{\i}(1-\Ga_{\t})=\Ga, \]
  where $\Ga$ is the attractor of the self-similar IFS in (\ref{eq:IFS-Ga}) and $\Ga_{\t}$ is the union given in (\ref{eq:Ga-t}).
\end{definition}

For a translation vector $\t=(t_0, t_1,\ldots, t_m)\in\R^{m+1}$ with $0=t_0<t_1<\cdots<t_m$ we define its \emph{conjugate} $\widehat\t=(\widehat t_0, \widehat t_1,\ldots, \widehat t_m)$ by letting $\widehat t_j = t_m - t_{m-j}$ for $0\le j \le m$.
Then we have $0=\widehat t_0 < \widehat t_1 < \cdots < \widehat t_m$.
Note by the symmetry of $\Ga$ that \[ (1+t_m)-\Ga_{\t} =\bigcup_{j=0}^m \big( (1-\Ga)+(t_m-t_j) \big)=\bigcup_{j=0}^m\big( \Ga+\widehat t_j \big)=\Ga_{\widehat\t}. \]
It follows that $\Ga_{\t}$ is a self-similar set if and only if $\Ga_{\widehat\t}$ is a self-similar set.
Based on the definition of admissible translation vectors, we can characterize the self-similarity of the union $\Ga_{\t}$.

\begin{theorem}\label{thm:union-self-similar}
Given $N \in \N$ and $0< \beta < 1/(N+1)$, let $\Ga$ be the attractor of the self-similar IFS in (\ref{eq:IFS-Ga}).
For $\t=(t_0,t_1,\ldots, t_m)\in\R^{m+1}$ with $0=t_0<t_1<\cdots<t_m$, the union $\Ga_{\t}$ defined in (\ref{eq:Ga-t}) is a self-similar set if and only if either $\t$ or its conjugate $\widehat\t$ is an admissible translation vector.
\end{theorem}
\begin{remark}
  The proof of Theorem \ref{thm:union-self-similar} relies on a characterization of all contractive similitudes $f$ on $\R$ satisfying $f(\Ga) \subset \Ga$ (see Lemma \ref{lemma:Gamma-self-embedding}). This approach differs from the arguments used in \cite[Theorem 1.5]{KLWYZ-2024} for $0 < \beta < 1/(2N+1)$.
  It is worth mentioning that the definition of admissible translation vectors slightly differs from the one in \cite[Definition 1.2]{KLWYZ-2024}, but they are equivalent for $0< \beta < 1/(2N+1)$.
\end{remark}

Next, we give a more practical criterion for admissible translation vectors via an associated directed graph.
Let $\mathcal{B}_N$ denote the set of $0< \beta < 1/(N+1)$ such that for any $n \in \N$, any $i_1,\ldots, i_n \in \{0,1,\ldots, 2N\}$, and any $j_1,\ldots,j_n \in \{0,1,\ldots,N\}$,
\[ \text{the equality}\;\;\sum_{k=1}^{n} i_k \beta^k = \sum_{k=1}^{n} j_k \beta^k \;\;\text{implies}\;\; i_k = j_k \;\;\forall 1 \le k \le n. \]
It is easy to check that $\big( 0, 1/(2N+1) \big] \subset \mathcal{B}_N$, and all transcendental numbers in $\big( 0, 1/(N+1) \big)$ are contained in $\mathcal{B}_N$.
Note that $\big( 0, 1/(N+1) \big) \setminus \mathcal{B}_N$ is countable and is dense in $\big( 1/(2N+1), 1/(N+1) \big)$.
For $n \in \N$, define $\beta_n$ to be the unique root in $\big( 0, 1/(N+1) \big)$ of the equation \[ N \sum_{k=1}^{n} x^k + 2 N \sum_{k=n+1}^{\f} x^k = 1, \]
i.e., $(N+1)x + N x^{n+1}=1$.
Note that $\beta_n \notin \mathcal{B}_N$ and the sequence $\{\beta_n\}_{n=1}^\f$ is increasing to $1/(N+1)$ as $n \to \f$.
Define the mapping $\vartheta: \big(0, 1/(N+1) \big) \to \N \cup \{0\}$ by
\[
\vartheta(\beta):=
\begin{cases}
  0, & \mbox{if } \beta \in \mathcal{B}_N, \\
  \min\big\{ n \in \N: \beta \le \beta_n \big\}, & \mbox{otherwise}.
\end{cases}
\]

Fix a translation vector $\t =(t_0, t_1, \ldots, t_m) \in T^{m+1}$ with $0=t_0<t_1<\cdots<t_m$.
For $n \ge \tau_{\t}$, since $\big\{ \phi_\i(t_j):\; \i\in\Omega_{\t}^n,\; 0\le j\le m \big\} \subset E_n$ and $\big\{\phi_\i(-t_j): \; \i\in\widehat\Omega_{\t}^n,\; 0\le j\le m \big\} \subset E_n$, there exist two subsets $\mathcal{A}_{\t}^n,\widehat{\mathcal{A}}_{\t}^n \subset \Omega^n$ such that
\begin{align*}
  \big\{ \phi_{\j}(0): \j \in  \mathcal{A}_{\t}^n \big\} & = \big\{ \phi_\i(t_j):\; \i\in\Omega_{\t}^n,\; 0\le j\le m \big\}, \\
  \big\{ \phi_{\j}(0): \j \in \widehat{\mathcal{A}}_{\t}^n \big\} & = \big\{\phi_\i(-t_j): \; \i\in\widehat\Omega_{\t}^n,\; 0\le j\le m \big\},
\end{align*}
and then define
\begin{equation*}%\label{eq:W}
  \mathcal{W}_{\t}^n := \mathcal{A}_{\t}^n \cup \widehat{\mathcal{A}}_{\t}^n.
\end{equation*}
We construct the directed graph $G_{\t}=(V_{\t}, E_{\t})$ as follows:
let $V_{\t} = \Omega^{\tau_{\t}+ \vartheta(\beta)} \setminus W_{\t}^{\tau_{\t}+ \vartheta(\beta)}$;
for $\i=i_1 i_2 \ldots i_{\tau_{\t}+ \vartheta(\beta)},\; \j=j_1 j_2 \ldots j_{\tau_{\t}+ \vartheta(\beta)} \in V_{\t}$, the edge from $\i$ to $\j$ belongs to $E_{\t}$ if and only if $i_2 i_3\ldots i_{\tau_{\t}+ \vartheta(\beta)} = j_1 j_2 \ldots j_{\tau_{\t}+ \vartheta(\beta) -1}$, that is, the suffix of $\i$ coincides with the prefix of $\j$.
Note that the directed graph $G_{\t}$ depends not only on the translation vector $\t$ but also on $\beta$.

A directed graph $G$ is said to contain a \emph{cycle} if there exists a directed path in $G$ starting and ending at the same vertex.
For convenience, the empty graph contains no cycles.

\begin{theorem}\label{thm:graph-criterion}
  A vector $\t =(t_0, t_1, \ldots, t_m) \in T^{m+1}$ with $0=t_0<t_1<\cdots<t_m$ is an admissible translation vector if and only if the directed graph $G_{\t}=(V_{\t}, E_{\t})$ contains no cycles.
\end{theorem}
\begin{remark}
  Compared with \cite[Proposition 1.3]{KLWYZ-2024}, the construction of the directed graph $G_{\t}$ differs slightly, but they coincide for $0 < \beta < 1/(2N+1)$.
  It is necessary to adapt the construction of $G_{\t}$ to ensure that Theorem \ref{thm:graph-criterion} is applicable for all $0 < \beta < 1/(N+1)$.
\end{remark}

By combining Theorem \ref{thm:graph-criterion} with Theorem \ref{thm:union-self-similar}, for any translation vector $\t \in \R^{m+1}$, we can, in theory, determine in finitely many steps whether the union $\Ga_{\t}$ is a self-similar set.
As an application, we obtain the following corollary, which extends \cite[Corollary 1.6]{KLWYZ-2024}.
For $x\in\R$ let $\lfloor x \rfloor$ denote its integer part.

\begin{corollary}\label{cor:m=1}
  Given $N \in \N$ and $\beta \in \mathcal{B}_N$, let $\Ga$ be the attractor of the self-similar IFS in (\ref{eq:IFS-Ga}).
  For $t>0$, the union $\Ga\cup(\Ga+t)$ is a self-similar set if and only if \[t = \frac{i(1-\beta)}{N} \beta^{-k}  \] for some $i \in \big\{ 1, 2, \ldots, \lfloor\frac{N+1}{2}\rfloor \big\}$ and $k\in\N$.
\end{corollary}
\begin{remark}
Note that $\big( 0, 1/(2N+1) \big] \subset \mathcal{B}_N$.
Thus, Corollary \ref{cor:m=1} fully recovers the result of Deng and Liu in \cite{Deng-Liu-2011}.
It is worth noting that Corollary \ref{cor:m=1} may fail for $\beta \not\in \mathcal{B}_N$ (see Examples \ref{example:not-in-B} and \ref{example-2:not-in-B}).
\end{remark}

\subsection{General self-similar sets}

For general self-similar sets, obtaining a simple criterion for the self-similarity of unions with translations is not straightforward. This primarily depends on the structure of self-similar sets.
If the self-similar set $K$ satisfies the SSC, then we can show that there are at most countable translation vectors $\t \in \R^{m+1}$ such that $K_{\t}$ is a self-similar set.

\begin{theorem}\label{thm:translation-countable}
  Let $K \subset \R$ be a self-similar set that satisfies the SSC.
  Then the set \[ \Big\{ (t_1, \ldots, t_m) \in \R^{m}: 0=t_0 < t_1 < \cdots < t_m, \;\text{and}\;\; \bigcup_{j=0}^m(K+t_j) \;\text{is a self-similar set} \Big\} \]
  is at most countable.
\end{theorem}

Next, we present two types of self-similar sets for which the union with their translations cannot be self-similar.

\begin{theorem}\label{thm:not-self-similar}
  {\rm(i)} For $\lambda_0,\lambda_1 >0$ with $\lambda_0+\lambda_1 < 1$, let $K \subset \R$ be the attractor of the self-similar IFS \[ \big\{ \varphi_{0}(x) = \lambda_0 x,\; \varphi_{1}(x) = \lambda_1 x + 1- \lambda_1 \big\}. \]
  If $\log \lambda_0 / \log \lambda_1 \notin \Q$, %then for any $t>0$, the union $K\cup (K+t)$ is not a self-similar set.
  then for any vector $\t=(t_0,t_1,\ldots, t_m)\in\R^{m+1}$ with $0=t_0<t_1<\cdots<t_m$, the union $K_{\t}=\bigcup_{j=0}^m (K+t_j)$ is not a self-similar set.

  {\rm(ii)} For $0< \lambda < 1/3$ and $\lambda < \zeta < 1-2\lambda$, let $K \subset \R$ be the attractor of the self-similar IFS \[ \big\{ \varphi_{0}(x) = \lambda x,\; \varphi_{1}(x) = \lambda x + \zeta,\;\varphi_2(x) = \lambda x + 1- \lambda \big\}. \]
  If $\zeta \not\in \Q(\lambda)$, then for any $t>0$, the union $K\cup (K+t)$ is not a self-similar set.
\end{theorem}

Finally, we give an example concerning inhomogeneous self-similar sets.
\begin{example}
For $0< \lambda < (\sqrt{5} -1)/2$, let $K$ be the attractor of the self-similar IFS
\[ \big\{ f_1(x) = \lambda x, f_2(x) = \lambda^2 x + 1 - \lambda^2 \big\}.  \]
Then the union $K \cup (K + \lambda^{-2} -1)$ is a self-similar set, which is the attractor of the self-similar IFS
\[ \big\{ g_1(x) = \lambda^2 x, \; g_2(x)=\lambda^3 x,\; g_3(x)=\lambda^2 x + \lambda^{-2} -1, \; g_4(x)=\lambda^3 x + \lambda^{-2} -1 \big\}. \]
\end{example}

The remainder of this paper is organized as follows.
In Section \ref{sec:2}, we provide a general criterion for the self-similarity of unions of self-similar sets and their translations and prove Theorem \ref{thm:union-self-similar}.
In Section \ref{sec:3}, we explore the equivalent characterization of admissible translation vectors via directed graph, and prove Theorem \ref{thm:graph-criterion}.
In Section \ref{sec:4}, we apply Theorems \ref{thm:union-self-similar} and \ref{thm:graph-criterion} to prove Corollary \ref{cor:m=1} and present some examples.
Finally, in Section \ref{sec:5}, we prove Theorems \ref{thm:translation-countable} and \ref{thm:not-self-similar}.

\section{Proof of Theorem \ref{thm:union-self-similar}}\label{sec:2}

In this section, we will prove Theorem \ref{thm:union-self-similar}.
We begin with a general criterion for the self-similarity of unions of self-similar sets and their translations.

Recall that $K \subset \R$ is a self-similar set if it is the attractor of a self-similar IFS $\mathcal{F}= \big\{ f_i(x) = \lambda_i x + a_i \big\}_{i=1}^\ell$, where $ 0 < |\lambda_i| < 1$ and $a_i \in \R$.
In this case, the set $K$ is also the attractor of the self-similar IFS $\mathcal{F}^{(n)} = \big\{ f_{i_1} \circ f_{i_2} \circ \cdots \circ f_{i_n}: 1 \le i_1, i_2, \ldots, i_n \le \ell\big\}$ for any $n \in \N$.
Thus, we can always assume that a self-similar set has a generating self-similar IFS with sufficiently small contractive ratios.
Note that if the IFS $\mathcal{F}$ satisfies the SSC, then $\mathcal{F}^{(n)}$ also satisfies the SSC for any $n \in \N$.

\begin{proposition}\label{prop:union-similarity-general}
  Let $K \subset \R$ be a self-similar set that satisfies the SSC, and let $\t=(t_0,t_1,\ldots, t_m)\in\R^{m+1}$ with $t_0< t_1 < \cdots < t_m$.
  Then the union $K_{\t}=\bigcup_{j=0}^m(K+t_j)$ is a self-similar set if and only if
  there exists a finite set $\mathcal G$ of similitudes such that
  \begin{equation}\label{eq:G-K}
    \bigcup_{g\in\mathcal G} g(K_{\t})=K.
  \end{equation}
\end{proposition}
\begin{proof}
  The sufficiency is straightforward. By (\ref{eq:G-K}), one can check that $K_{\t}$ is the attractor of the self-similar IFS $\{ g(x)+ t_j: g \in \mathcal{G},\; 0 \le j \le m \}$.
  We now turn to the necessity.

  Write $a= \min K$ and $b = \max K$.
  Let $\mathcal{F}=\{f_i(x) = \lambda_i x + a_i \}_{i=1}^{\ell}$ be the self-similar IFS satisfying the SSC such that $K$ is the attractor of $\mathcal{F}$.
  Note that $a,b \in K = \bigcup_{i=1}^\ell f_i(K)$.
  Without loss of generality, we may assume that $a \in f_1(K)$ and $b \in f_\ell(K)$.
  If $\lambda_1>0$ then we have $a = f_1(a)$.
  If $\lambda_\ell >0$ then we have $b = f_\ell (b)$.
  If $\lambda_1< 0$ and $\lambda_\ell < 0$, then we have $a=f_1(b)$ and $b = f_\ell(a)$.
  In this situation, instead of $\mathcal{F}$, consider the self-similar IFS $\mathcal{F}^{(2)}=\big\{ f_i \circ f_j: 1 \le i, j \le \ell \big\}$, where $\mathcal{F}^{(2)}$ satisfies the SSC, $K$ is the attractor of $\mathcal{F}^{(2)}$, and $a=f_1\circ f_\ell(a)$.
  Thus, we can assume that the IFS $\mathcal{F}$ satisfies either $a = f_1(a)$ or $b = f_\ell(b)$.
  The arguments for the two cases are essentially identical, so we restrict our attention to the case $a=f_1(a)$ in what follows.

  Choose a large integer $n \in \N$ such that $|\lambda_1|^n(b-a) < t_1 - t_0$, and let $E = f_1^n(K) + t_0$.
  Since $E \subset [a+t_0, a+|\lambda_1|^n(b-a)+t_0]$ and $\bigcup_{j=1}^m (K+t_j) \subset [a+t_1, +\f)$, we have
  \begin{equation}\label{eq:positive-dist}
    \mathrm{dist}\bigg( E , \bigcup_{j=1}^m (K+t_j) \bigg)>0.
  \end{equation}
  Note that the IFS $\mathcal{F}$ satisfies the SSC.
  We have $\mathrm{dist}\big( f_1^n(K), K \setminus f_1^n(K) \big)>0$.
  It follows that $\mathrm{dist}\big( E, (K+t_0) \setminus E \big)>0$.
  Together with (\ref{eq:positive-dist}), we conclude that $\mathrm{dist}(E, K_{\t} \setminus E)>0$.

  Suppose that $K_{\t}$ is a self-similar set, i.e., the attractor of the self-similar IFS $\big\{ \varphi_i(x) = r_i x + b_i \big\}_{i=1}^{q}$.
  Let $r = \max\{|r_1|, |r_2|, \ldots, |r_q|\}$.
  We may assume that $r$ is sufficiently small so that
  \begin{equation}\label{eq:r-small}
    r < \frac{\mathrm{dist}(E, K_{\t} \setminus E)}{b-a+t_m - t_0}.
  \end{equation}
  Let $\Lambda = \big\{ 1 \le i \le q: \varphi_i(K_{\t}) \cap E \ne \emptyset \big\}$.
  Note that $E \subset K_{\t} = \bigcup_{i=1}^q \varphi_i(K_{\t})$.
  We have
  \begin{equation}\label{eq:e-cover}
    E \subset \bigcup_{i \in \Lambda} \varphi_i(K_{\t}).
  \end{equation}
  For any $1 \le i \le q$, by (\ref{eq:r-small}) the diameter of $\varphi_i(K_{\t})$ is strictly less than $\mathrm{dist}(E, K_{\t} \setminus E)$, which means that $\varphi_i(K_{\t})$ cannot intersect both $E$ and $K_{\t} \setminus E$ simultaneously.
  For $i \in \Lambda$, since $\varphi_i(K_{\t}) \cap E \ne \emptyset$, we conclude that $\varphi_i(K_{\t}) \subset E$.
  Together with (\ref{eq:e-cover}) we obtain \[ E = \bigcup_{i \in \Lambda} \varphi_i(K_{\t}). \]
  Write $\varphi(x) = f_1^n(x) + t_0$.
  Note that $E = \varphi(K)$.
  Thus the set $\mathcal{G}=\{ \varphi^{-1} \circ \varphi_i: i\in \Lambda\}$ satisfies (\ref{eq:G-K}).
  The proof is complete.
\end{proof}

Before we turn to the proof of Theorem \ref{thm:union-self-similar}, we present a corollary of Proposition \ref{prop:union-similarity-general}.

\begin{corollary}\label{cor:scale}
  Let $K \subset \R$ be the attractor of a homogeneous self-similar set IFS $\mathcal{F} = \big\{ f_i(x) = \lambda x + a_i\big\}_{i=1}^\ell$ that satisfies the SSC, where $0< \lambda <1$ and $a_i \in \R$.
  Let $\t=(t_0,t_1,\ldots, t_m)\in\R^{m+1}$ with $0=t_0< t_1 < \cdots < t_m$.
  If $K_{\t}=\bigcup_{j=0}^m(K+t_j)$ is a self-similar set, then $K_{\lambda^{-k} \t} = \bigcup_{j=0}^m(K+\lambda^{-k}t_j)$ is also a self-similar set for any $k \in \N$.
\end{corollary}
\begin{proof}
  Fix $k \in \N$. Note that $K_{\lambda^{-k} \t} \ne \lambda^{-k} K_{\t}$. The desired conclusion is not immediately obvious.
  Let $\Omega = \{1,2,\ldots, \ell\}$, and for $\i=i_1i_2\ldots i_n \in \Omega^n$ define $f_{\i} = f_{i_1} \circ f_{i_2} \circ \cdots \circ f_{i_n}$.
  Note that \[ K = \bigcup_{\i \in \Omega^k} f_{\i}(K). \]
  It follows that
  \begin{equation}\label{eq:scale-general}
    \bigcup_{\i \in \Omega^k} f_{\i}\big( K_{\lambda^{-k} \t} \big)
    = \bigcup_{j=0}^m \bigcup_{\i \in \Omega^k} f_{\i} \big( K+\lambda^{-k}t_j \big)
    = \bigcup_{j=0}^m \bigcup_{\i \in \Omega^k} \big( f_{\i}(K)+t_j \big)
    = K_{\t}.
  \end{equation}
  If $K_{\t}$ is a self-similar set, then by Proposition \ref{prop:union-similarity-general} there exists a finite set $\mathcal G$ of similitudes such that $\bigcup_{g\in\mathcal G} g(K_{\t})=K$.
  It follows from (\ref{eq:scale-general}) that
  \[ \bigcup_{g\in\mathcal G} \bigcup_{\i \in \Omega^k} g \circ f_{\i} \big( K_{\lambda^{-k} \t} \big)=K. \]
  Again by Proposition \ref{prop:union-similarity-general}, we conclude that $K_{\lambda^{-k} \t}$ is also a self-similar set.
\end{proof}

Now we focus on homogeneous symmetric self-similar sets $\Ga = \Ga_{\beta, \{0,1,\ldots,N\}}$, which is the attractor of the self-similar IFS  \[ \Big\{ \phi_i(x)=\beta x + i \frac{1-\beta}{N}: i=0,1,\ldots, N \Big\},\] where $N \in \N$ and $0< \beta < 1/(N+1)$.
Recall that $\Omega=\{0,1,\ldots,N\}$, and for $\i= i_1 i_2 \ldots i_n \in \Omega^n$ define $\phi_{\i}=\phi_{i_1} \circ \phi_{i_2} \circ \cdots \circ \phi_{i_n}$.
The following lemma characterizes all contractive similitudes $f$ on $\R$ such that $f(\Gamma) \subset \Gamma$.

\begin{lemma}{\rm\cite[Corollary 1.3]{Wang-2025}}\label{lemma:Gamma-self-embedding}
  If $f$ is a contractive similitude on $\R$ such that $f(\Gamma) \subset \Gamma$, then there exists $\i \in \bigcup_{n=1}^\f \Omega^n$ such that $f(x) = \phi_{\i}(x)$ or $f(x) = \phi_{\i}(1-x)$.
\end{lemma}

By combining Lemma \ref{lemma:Gamma-self-embedding} with Proposition \ref{prop:union-similarity-general}, we immediately obtain the following proposition.

\begin{proposition}\label{prop:Gamma-self-similarity}
  Let $\t=(t_0,t_1,\ldots, t_m)\in\R^{m+1}$ with $0=t_0< t_1 < \cdots < t_m$.
  Then the union $\Ga_{\t}=\bigcup_{j=0}^m(\Ga+t_j)$ is a self-similar set if and only if there exist two finite sets $\I_1, \I_2\subset \bigcup_{n=1}^\f \Omega^n$ such that
  \[ \bigcup_{\i\in \I_1}\phi_{\i}(\Ga_{\t})\cup\bigcup_{\i\in\I_2}\phi_{\i}(1-\Ga_{\t})=\Ga. \]
\end{proposition}
\begin{proof}
  The sufficiency follows directly from the sufficient part of Proposition \ref{prop:union-similarity-general}.
  For the necessity, we assume that the union $\Ga_{\t}$ is a self-similar set.
  By Proposition \ref{prop:union-similarity-general}, there exists a finite set $\mathcal G$ of similitudes such that $\bigcup_{g\in\mathcal G} g(\Ga_{\t})=\Ga$. Note that $\Gamma \subset \Ga_{\t}$, so for each $g\in\mathcal G$ we have $g(\Ga) \subset \Ga$.
  By Lemma \ref{lemma:Gamma-self-embedding}, each $g\in\mathcal G$ has the form $\phi_{\i}(x)$ or $\phi_{\i}(1-x)$ for some $\i \in \bigcup_{n=1}^\f \Omega^n$.
  Thus, the two finite sets \[ \mathcal{I}_1 = \bigg\{ \i \in \bigcup_{n=1}^\f \Omega^n: \phi_{\i}(x) \in \mathcal{G}  \bigg\} \quad\text{and}\quad \mathcal{I}_2 = \bigg\{ \i \in \bigcup_{n=1}^\f \Omega^n: \phi_{\i}(1-x) \in \mathcal{G}  \bigg\} \] are the desired sets.
\end{proof}

Recall that $E_n = \big\{ \phi_{\i}(0): \i \in \Omega^n \big\}$ and $T_n = \beta^{-n} E_n$ for $n \in \N$, and \[ T = \bigcup_{n=1}^\f T_n. \]
For a translation vector $\t=(t_0, t_1,\ldots, t_m)\in\R^{m+1}$ with $0=t_0<t_1<\cdots<t_m$, its conjugate $\widehat\t=(\widehat t_0, \widehat t_1,\ldots, \widehat t_m)$ is defined by letting $\widehat t_j = t_m - t_{m-j}$ for $0\le j \le m$.

\begin{proposition}\label{prop:translation}
  Let $\t=(t_0,t_1,\ldots, t_m)\in\R^{m+1}$ with $0=t_0< t_1 < \cdots < t_m$.
  If the union $\Ga_{\t}=\bigcup_{j=0}^m(\Ga+t_j)$ is a self-similar set, then we have either $\t \in T^{m+1}$ or $\widehat\t \in T^{m+1}$.
\end{proposition}
\begin{proof}
  By Proposition \ref{prop:Gamma-self-similarity}, there exist two finite sets $\I_1, \I_2\subset \bigcup_{n=1}^\f \Omega^n$ such that
  \begin{equation}\label{eq:I-1-I-2-Gamma}
    \bigcup_{\i\in \I_1}\phi_{\i}(\Ga_{\t})\cup\bigcup_{\i\in\I_2}\phi_{\i}(1-\Ga_{\t})=\Ga.
  \end{equation}
  Note that $0 \in \Ga$.
  By (\ref{eq:I-1-I-2-Gamma}), there are only two possible cases: (i) $0 \in \phi_{\i}(\Ga_{\t})$ for some $\i\in \I_1$; (ii) $0 \in \phi_{\i}(1-\Ga_{\t})$ for some $\i\in \I_2$.

  \textbf{Case (i)}: $0 \in \phi_{\i}(\Ga_{\t})$ for some $\i\in \I_1$. Note that $0$ is the minimal value of $\Ga$.
  We have $0 = \min \phi_{\i}(\Gamma_{\t}) = \phi_{\i}(0)$. It follows that $\phi_{\i}(x) = \beta^n x$, where $n$ is the length of $\i$.
  For $0\le j \le m$, since $\phi_{\i}(\Ga + t_j) \subset \Ga_{\t} \subset \Gamma$, by Lemma \ref{lemma:Gamma-self-embedding} there exists $\mathbf{j} \in \bigcup_{k=1}^\f \Omega^k$ (depending on $j$) such that $\phi_{\i}(x +t_j) = \phi_{\mathbf{j}}(x)$.
  It follows that $\mathbf{j} \in \Omega^n$ and $\phi_{\mathbf{j}}(0) = \phi_{\i}(t_j)=\beta^n t_j$.
  So we have $t_j = \beta^{-n} \phi_{\mathbf{j}}(0) \in T_n \subset T$.
  Thus, we obtain that $\t \in T^{m+1}$.

  \textbf{Case (ii)}: $0 \in \phi_{\i}(1-\Ga_{\t})$ for some $\i\in \I_2$.
  Since $0$ is the minimal value of $\Ga$, we have $0 = \min \phi_{\i}(1-\Gamma_{\t}) = \phi_{\i}(-t_m)$. That is, $\phi_{\i}(0) = \beta^n t_m$ where $n$ is the length of $\i$.
  For $0\le j \le m$, since $\phi_{\i}(\Ga - t_j)= \phi_{\i}\big( 1- (\Ga + t_j) \big) \subset \phi_{\i}(1-\Ga_{\t}) \subset \Gamma$, by Lemma \ref{lemma:Gamma-self-embedding} there exists $\mathbf{j} \in \bigcup_{k=1}^\f \Omega^k$ (depending on $j$) such that $\phi_{\i}(x - t_j) = \phi_{\mathbf{j}}(x)$.
  It follows that $\mathbf{j} \in \Omega^n$ and $\phi_{\mathbf{j}}(0) = \phi_{\i}(-t_j) = \phi_{\i}(0) - \beta^n t_j = \beta^n(t_m - t_j)$.
  So $\widehat{t_j} = t_m - t_j = \beta^{-n} \phi_{\mathbf{j}}(0) \in T_n \subset T$.
  Thus, we obtain that $\widehat\t \in T^{m+1}$.
  The proof is complete.
\end{proof}

We are very close to completing the proof of Theorem \ref{thm:union-self-similar}.
Recall that for a translation vector $\t =(t_0, t_1, \ldots, t_m) \in T^{m+1}$ with $0=t_0<t_1<\cdots<t_m$, $\tau_{\t} \in \N$ is the smallest integer such that $ t_1, \ldots, t_m \in T_{\tau_{\t}}$.
For $n\ge \tau_{\t}$, we define
\begin{align*}
  \Omega_{\t}^n & = \big\{ \i \in \Omega^n:  \phi_{\i}(t_j) \in E_n \;\; \forall 1\le j \le m  \big\}, \\
  \widehat{\Omega}_{\t}^n & = \big\{ \i \in \Omega^n:  \phi_{\i}(-t_j) \in E_n \;\; \forall 1\le j \le m  \big\}.
\end{align*}

\begin{proof}[Proof of Theorem \ref{thm:union-self-similar}]
  If $\t$ or $\widehat\t$ is an admissible translation vector, then by Proposition \ref{prop:Gamma-self-similarity}, the corresponding set $\Ga_{\t}$ or $\Ga_{\widehat\t}$ is a self-similar set. Thus we conclude that $\Ga_{\t}$ is a self-similar set.
  This establishes the sufficiency.
  We now focus on the necessity.
  Suppose that $\Ga_{\t}$ is a self-similar set.
  By Proposition \ref{prop:translation}, we have either $\t \in T^{m+1}$ or $\widehat\t \in T^{m+1}$.

  We first consider the case $\t \in T^{m+1}$.
  Since $\Ga_{\t}$ is a self-similar set, by Proposition \ref{prop:Gamma-self-similarity}, there exist two finite sets $\I_1, \I_2\subset \bigcup_{n=1}^\f \Omega^n$ such that
  \begin{equation*}%\label{eq:I-1-I-2-Gamma-2}
    \bigcup_{\i\in \I_1}\phi_{\i}(\Ga_{\t})\cup\bigcup_{\i\in\I_2}\phi_{\i}(1-\Ga_{\t})=\Ga.
  \end{equation*}
  To verify that $\t$ is an admissible translation vector, it remains to show that $\I_1 \subset \bigcup_{n=\tau_{\t}}^\f \Omega_{\t}^n$ and $\I_2 \subset \bigcup_{n=\tau_{\t}}^\f \widehat\Omega_{\t}^n$.

  By the minimality of $\tau_{\t}$, there exists $1\le j_0 \le m$ such that $t_{j_0} \in T_{\tau_{\t}} \setminus T_{\tau_{\t}-1}$, where $T_0:=\{0\}$. This implies that
  \begin{equation}\label{eq:t-j-0}
    t_{j_0} \ge \frac{1-\beta}{N \beta^{\tau_{\t}}}.
  \end{equation}

  Take any $\i \in \mathcal{I}_1$ with length $n$.
  For any $1\le j \le m$, since $\phi_{\i}(\Ga + t_j) \subset \phi_{\i}(\Ga_{\t}) \subset \Gamma$, by Lemma \ref{lemma:Gamma-self-embedding} there exists $\mathbf{j} \in \bigcup_{k=1}^\f \Omega^k$ (depending on $j$) such that $\phi_{\i}(x +t_j) = \phi_{\mathbf{j}}(x)$.
  It follows that $\mathbf{j} \in \Omega^n$ and $\phi_{\i}(t_j)= \phi_{\mathbf{j}}(0)$.
  Thus, we obtain that $\phi_{\i}(t_j) \in E_n$ for any $1 \le j \le m$.
  Note that $\phi_{\i}(t_{j_0})= \beta^n t_{j_0} + \phi_{\i}(0) \in E_n\subset [0,1]$.
  It follows from (\ref{eq:t-j-0}) that \[ \frac{1-\beta}{N \beta^{\tau_{\t}-n}} \le \beta^n t_{j_0} \le 1. \]
  This implies that $n \ge \tau_{\t}$ because $\beta < 1/(N+1)$.
  So we conclude that $\i \in \Omega_{\t}^n$ with $n \ge \tau_{\t}$.
  Thus, we obtain that $\I_1 \subset \bigcup_{n=\tau_{\t}}^\f \Omega_{\t}^n$.

  Take any $\i \in \mathcal{I}_2$ with length $n$.
  For any $1\le j \le m$, since $\phi_{\i}(\Ga - t_j) = \phi_{\i}\big( 1-(\Ga+t_j) \big) \subset \phi_{\i}(1-\Ga_{\t}) \subset \Gamma$, by Lemma \ref{lemma:Gamma-self-embedding} there exists $\mathbf{j} \in \bigcup_{k=1}^\f \Omega^k$ (depending on $j$) such that $\phi_{\i}(x -t_j) = \phi_{\mathbf{j}}(x)$.
  It follows that $\mathbf{j} \in \Omega^n$ and $\phi_{\i}(-t_j)= \phi_{\mathbf{j}}(0)\in E_n$.
  Thus, we obtain that $\phi_{\i}(-t_j) \in E_n$ for any $1 \le j \le m$.
  Note that $\phi_{\i}(-t_{j_0})= \phi_{\i}(0) - \beta^n t_{j_0} \in E_n \subset [0,1]$.
  It follows from (\ref{eq:t-j-0}) that \[ \frac{1-\beta}{N \beta^{\tau_{\t}-n}} \le \beta^n t_{j_0} \le \phi_{\i}(0) \le 1. \]
  This implies that $n \ge \tau_{\t}$ because $\beta < 1/(N+1)$.
  So we conclude that $\i \in \widehat{\Omega}_{\t}^n$ with $n \ge \tau_{\t}$.
  Thus, we obtain that $\I_2 \subset \bigcup_{n=\tau_{\t}}^\f \widehat\Omega_{\t}^n$.

  Therefore, we have proved that $\t$ is an admissible translation vector.
  For the case $\widehat\t \in T^{m+1}$, observe that $\Gamma_{\widehat{t}}$ is also a self-similar set, and by the same argument as above, $\widehat{\t}$ is seen to be an admissible translation vector.
  This verifies the necessity.
  The proof is complete.
\end{proof}

\section{Proof of Theorem \ref{thm:graph-criterion}}\label{sec:3}

In this section, we present a criterion for admissible translation vectors using directed graphs, and prove Theorem \ref{thm:graph-criterion}.

Recall that $\Ga = \Ga_{\beta, \{0,1,\ldots,N\}}$ is the attractor of the self-similar IFS  \[ \Big\{ \phi_i(x)=\beta x + i \frac{1-\beta}{N}: i=0,1,\ldots, N \Big\},\] where $N \in \N$ and $0< \beta < 1/(N+1)$.
Let $\Omega=\{0,1,\ldots,N\}$, and for $\i= i_1 i_2 \ldots i_n \in \Omega^n$ define $\phi_{\i}=\phi_{i_1} \circ \phi_{i_2} \circ \cdots \circ \phi_{i_n}$.
Note that for any $n \in \N$, \[ \Gamma = \bigcup_{\i \in \Omega^n} \phi_{\i}(\Ga), \]
where the union on the right-hand side is pairwise disjoint.

Recall that $E_n = \big\{ \phi_{\i}(0): \i \in \Omega^n \big\}$ and $T_n = \beta^{-n} E_n$ for $n \in \N$, and \[ T = \bigcup_{n=1}^\f T_n. \]
In this section, we always assume that the translation vector $\t=(t_0, t_1 , \ldots , t_m) \in T^{m+1}$ with $0=t_0 < t_1 < \cdots < t_m$, and $\Ga_{\t} = \bigcup_{j=0}^m (\Ga + t_j)$.

Recall that $\tau_{\t} \in \N$ is the smallest integer such that $t_1, \ldots, t_m \in T_{\tau_{\t}}$.
For $n\ge \tau_{\t}$, let
\begin{align*}
  \Omega_{\t}^n & = \big\{ \i \in \Omega^n:  \phi_{\i}(t_j) \in E_n \;\; \forall 1\le j \le m  \big\}, \\
  \widehat{\Omega}_{\t}^n & = \big\{ \i \in \Omega^n:  \phi_{\i}(-t_j) \in E_n \;\; \forall 1\le j \le m  \big\},
\end{align*}
and let $\mathcal{A}_{\t}^n,\widehat{\mathcal{A}}_{\t}^n \subset \Omega^n$ such that
\begin{align}
  \big\{ \phi_{\j}(0): \j \in  \mathcal{A}_{\t}^n \big\} & = \big\{ \phi_\i(t_j):\; \i\in\Omega_{\t}^n,\; 0\le j\le m \big\}, \label{eq:A-t-n}\\
  \big\{ \phi_{\j}(0): \j \in \widehat{\mathcal{A}}_{\t}^n \big\} & = \big\{\phi_\i(-t_j): \; \i\in\widehat\Omega_{\t}^n,\; 0\le j\le m \big\}. \label{eq:hat-A-t-n}
\end{align}
Define
\begin{equation*}%\label{eq:W}
  \mathcal{W}_{\t}^n = \mathcal{A}_{\t}^n \cup \widehat{\mathcal{A}}_{\t}^n.
\end{equation*}
We first characterize admissible translation vectors using the sets $\mathcal{W}_{\t}^n,\; n \ge \tau_{\t}$.

\begin{proposition}\label{prop:symbol-criterion}
  A vector $\t=(t_0, t_1 , \ldots , t_m) \in T^{m+1}$ with $0=t_0 < t_1 < \cdots < t_m$ is an admissible translation vector if and only if there exists $\ell \ge \tau_{\t}$ such that
  \[ \bigcup_{n=\tau_{\t}}^\ell \mathcal{W}_{\t}^n \times \Omega^{\ell-n} = \Omega^\ell. \]
\end{proposition}
\begin{proof}
  Note that for $n \ge \tau_{\t}$, by (\ref{eq:A-t-n}) we have
  \begin{align*}
    \bigcup_{\i \in \Omega_{\t}^n} \phi_{\i}(\Ga_{\t})
    & = \bigcup_{\i \in \Omega_{\t}^n} \bigcup_{j=0}^m \phi_{\i}(\Ga+t_j) = \bigcup_{\i \in \Omega_{\t}^n} \bigcup_{j=0}^m \Big( \beta^n \Ga + \phi_{\i}(t_j) \Big) \\
    & = \bigcup_{\j \in \mathcal{A}_{\t}^n} \Big( \beta^n \Ga + \phi_{\j}(0) \Big) = \bigcup_{\j \in \mathcal{A}_{\t}^n} \phi_{\j}(\Ga).
  \end{align*}
  Similarly, by (\ref{eq:hat-A-t-n}) we have \[ \bigcup_{\i \in \widehat\Omega_{\t}^n} \phi_{\i}(1-\Ga_{\t}) = \bigcup_{\j \in \widehat{\mathcal{A}}_{\t}^n} \phi_{\j}(\Ga). \]
  It follows that for $n \ge \tau_{\t}$,
  \begin{equation}\label{eq:Omega-to-W}
    \bigcup_{\i \in \Omega_{\t}^n} \phi_{\i}(\Ga_{\t}) \cup \bigcup_{\i \in \widehat\Omega_{\t}^n} \phi_{\i}(1-\Ga_{\t}) = \bigcup_{\j \in \mathcal{W}_{\t}^n} \phi_{\j}(\Ga).
  \end{equation}

  For the sufficiency, let $\I_1 = \bigcup_{n=\tau_{\t}}^\ell \Omega_{\t}^n$ and $\I_2 = \bigcup_{n=\tau_{\t}}^\ell \widehat\Omega_{\t}^n$. By (\ref{eq:Omega-to-W}) we obtain that
  \begin{align*}
    \bigcup_{\i \in \I_1} \phi_{\i}(\Ga_{\t}) \cup \bigcup_{\i \in \I_2} \phi_{\i}(1-\Ga_{\t})
    & = \bigcup_{n = \tau_{\t}}^\ell \bigcup_{\j \in \mathcal{W}_{\t}^n} \phi_{\j}(\Ga) \\
    & = \bigcup_{n = \tau_{\t}}^\ell \bigcup_{\j \in \mathcal{W}_{\t}^n} \bigcup_{\i \in \Omega^{\ell-n}} \phi_{\j\i}(\Ga) \\
    & = \bigcup_{\i \in \Omega^\ell} \phi_{\i}(\Ga) = \Ga.
  \end{align*}
  Thus, by Definition \ref{def:admissible-translation}, $\t$ is an admissible translation vector.

  Next, we prove the necessity. Assume that $\t$ is an admissible translation vector.
  By Definition \ref{def:admissible-translation}, there exist two finite sets $\I_1\subset \bigcup_{n=\tau_{\t}}^\f \Omega_{\t}^n$ and $\I_2 \subset \bigcup_{n=\tau_{\t}}^\f \widehat{\Omega}_{\t}^n$ such that
  \begin{equation}\label{eq:I-1-I-2-Gamma-3}
    \bigcup_{\i\in \I_1}\phi_{\i}(\Ga_{\t})\cup\bigcup_{\i\in\I_2}\phi_{\i}(1-\Ga_{\t})=\Ga.
  \end{equation}
  Let $\ell$ be the longest length of words in $\I_1 \cup \I_2$.
  Clearly, we have $\I_1 \subset \bigcup_{n=\tau_{\t}}^\ell \Omega_{\t}^n$ and $\I_2 \subset \bigcup_{n=\tau_{\t}}^\ell \widehat\Omega_{\t}^n$.
  It follows from (\ref{eq:I-1-I-2-Gamma-3}) and (\ref{eq:Omega-to-W}) that
  \[ \Ga \subset \bigcup_{n=\tau_{\t}}^\ell \Big(\bigcup_{ \i \in \Omega_{\t}^n} \phi_{\i}(\Gamma_{\t}) \cup \bigcup_{ \i \in \widehat{\Omega}_{\t}^n } \phi_{\i}(1-\Gamma_{\t}) \Big) = \bigcup_{n=\tau_{\t}}^\ell \bigcup_{\j \in \mathcal{W}_{\t}^n} \phi_{\j}(\Ga).  \]
  The inverse inclusion is obvious. Thus, we obtain that
  \[ \Ga = \bigcup_{n=\tau_{\t}}^\ell \bigcup_{\j \in \mathcal{W}_{\t}^n} \phi_{\j}(\Ga). \]
  That is, \[ \bigcup_{\i' \in \Omega^\ell} \phi_{\i'}(\Ga) = \bigcup_{n=\tau_{\t}}^\ell \bigcup_{\j \in \mathcal{W}_{\t}^n} \bigcup_{\i \in \Omega^{\ell-n}} \phi_{\j\i}(\Ga). \]
  Therefore, we conclude that \[ \bigcup_{n=\tau_{\t}}^\ell \mathcal{W}_{\t}^n \times \Omega^{\ell-n} = \Omega^\ell, \] as desired.
\end{proof}

Next, we investigate the relation between the sets $\mathcal{W}_{\t}^{n},\;n \ge \tau_{\t}$.

\begin{lemma}\label{lemma:W-subset}
  For $n \ge \tau_{\t}$, we have
  \[ \Omega \times \mathcal{W}_{\t}^{n} \subset \mathcal{W}_{\t}^{n+1}.\]
\end{lemma}
\begin{proof}
  Fix $n \ge \tau_{\t}$. Take any $\j \in \mathcal{A}_{\t}^n$, and
  there exist $\i \in \Omega_{\t}^n$ and $0\le j_0 \le m$ such that $\phi_{\j}(0) = \phi_{\i}(t_{j_0})$.
  Since $\i \in \Omega_{\t}^n$, we have $\phi_{\i}(t_k) \in E_n$ for all $1\le k \le m$.
  For any $i_0 \in \Omega$, note that \[ \phi_{i_0\i}(t_k) \in \phi_{i_0}(E_n) \subset E_{n+1}\quad \forall 1\le k \le m, \]
  and hence, we obtain $i_0 \i \in \Omega_{\t}^{n+1}$.
  Note that $\phi_{i_0 \j}(0) = \phi_{i_0 \i}(t_{j_0})$.
  Thus, we have $i_0 \j \in \mathcal{A}_{\t}^{n+1}$ for any $i_0 \in \Omega$.
  It follows that $\Omega \times \mathcal{A}_{\t}^n \subset \mathcal{A}_{\t}^{n+1}$.

  Similarly, we can show that $\Omega \times \widehat{\mathcal{A}}_{\t}^n \subset \widehat{\mathcal{A}}_{\t}^{n+1}$.
  Note that $\mathcal{W}_{\t}^n = \mathcal{A}_{\t}^n \cup \widehat{\mathcal{A}}_{\t}^n$. Thus, we conclude that $\Omega \times \mathcal{W}_{\t}^{n} \subset \mathcal{W}_{\t}^{n+1}$.
 
  We complete the proof.
\end{proof}

Recall that $\mathcal{B}_N$ is the set of $0< \beta < 1/(N+1)$ such that for any $n \in \N$, any $i_1,\ldots, i_n \in \{0,1,\ldots, 2N\}$, and any $j_1,\ldots, j_n \in \{0,1,\ldots,N\}$,
\[ \text{the equality}\;\;\sum_{k=1}^{n} i_k \beta^k = \sum_{k=1}^{n} j_k \beta^k \;\;\text{implies}\;\; i_k = j_k \;\;\forall 1 \le k \le n. \]
For $n \in \N$, $\beta_n$ is the unique root in $\big( 0, 1/(N+1) \big)$ of the equation \[ N \sum_{k=1}^{n} x^k + 2 N \sum_{k=n+1}^{\f} x^k = 1. \]
The mapping $\vartheta: \big(0, 1/(N+1) \big) \to \N \cup \{0\}$ is defined by
\[
\vartheta(\beta):=
\begin{cases}
  0, & \mbox{if } \beta \in \mathcal{B}_N, \\
  \min\big\{ n \in \N: \beta \le \beta_n \big\}, & \mbox{otherwise}.
\end{cases}
\]

\begin{lemma}\label{lemma:i-1-equal-j-1}
  Given $n > \tau_{\t}+\vartheta(\beta)$ and $0 \le j \le m$, for any $\i=i_1 i_2 \ldots i_n,  \j=j_1 j_2 \ldots j_n \in \Omega^n$, if $\phi_{\i}(t_{j}) = \phi_{\j}(0)$ then we have $i_1 = j_1$.
\end{lemma}
\begin{proof}
Since $t_j \in T_{\tau_{\t}}$, there exists a unique word $t_{j,1} t_{j,2} \ldots t_{j,\tau_{\t}} \in \Omega^{\tau_{\t}}$ such that $t_j = \beta^{-\tau_{\t}} \phi_{t_{j,1} t_{j,2} \ldots t_{j,\tau_{\t}}} (0)$. Note that $\phi_{\i}(t_{j}) = \phi_{\i}(0) + \beta^{n-\tau_{\t}} \phi_{t_{j,1} t_{j,2} \ldots t_{j,\tau_{\t}}} (0)= \phi_{\j}(0)$.
That is,
\begin{equation}\label{eq:i-equal-j}
  \sum_{k=1}^{n} i_k \beta^n + \beta^{n-\tau_{\t}} \sum_{k=1}^{\tau_{\t}} t_{j,k} \beta^k = \sum_{k=1}^{n} j_k \beta^k.
\end{equation}

Suppose on the contrary that $i_1 \ne j_1$. By the definition of $\mathcal{B}_N$, we obtain that $\beta \not\in \mathcal{B}_N$.
So we have $\vartheta(\beta)\ge 1$ and $\beta \le \beta_{\vartheta(\beta)}$.
Using the fact that $0 < \beta < 1/(N+1)$, it is easy to estimate that
\[ \sum_{k=1}^{n} i_k \beta^n + \beta^{n-\tau_{\t}} \sum_{k=1}^{\tau_{\t}} t_{j,k} \beta^k \in \Big[ i_1 \beta, (i_1+2) \beta \Big) \quad\text{and}\quad \sum_{k=1}^{n} j_k \beta^k \in \Big[ j_1 \beta, (j_1+1) \beta \Big) . \]
Since $j_1 \ne i_1$, by (\ref{eq:i-equal-j}) we obtain that $j_1 = i_1 + 1$.
It follows from (\ref{eq:i-equal-j}) that
\[ (i_1+1) \beta \le \sum_{k=1}^{n} i_{k} \beta^n + \beta^{n-\tau_{\t}} \sum_{k=1}^{\tau_{\t}} t_{j,k} \beta^k. \]
That is,
\[ 1 \le \sum_{k=1}^{n-1} i_{k+1} \beta^n + \beta^{n-\tau_{\t}-1} \sum_{k=1}^{\tau_{\t}} t_{j,k} \beta^k < N \sum_{k=1}^{n-\tau_{\t}-1} \beta^k + 2N \sum_{k=n-\tau_{\t}}^\f \beta^k. \]
This implies that $\beta > \beta_{n-\tau_{\t}-1}$.
Note that the sequence $\{\beta_k\}_{k=1}^\f$ is increasing and $\vartheta(\beta) \le n - \tau_{\t}-1$.
We obtain that $\beta > \beta_{\vartheta(\beta)}$,
which contradicts the fact that $\beta \le \beta_{\vartheta(\beta)}$.
Thus, we conclude that $i_1 = j_1$.
\end{proof}

\begin{proposition}\label{prop:W-equality}
  Write $n_0 = \tau_{\t}+\vartheta(\beta)$.
  Then for $n \ge n_0$, we have \[ \mathcal{W}_{\t}^n = \Omega^{n-n_0} \times \mathcal{W}_{\t}^{n_0}. \]
\end{proposition}
\begin{proof}
For $n \ge n_0$, by Lemma \ref{lemma:W-subset} we have $\Omega^{n-n_0} \times \mathcal{W}_{\t}^{n_0} \subset \mathcal{W}_{\t}^n$.
It suffices to show that $\mathcal{W}_{\t}^n \subset \Omega^{n-n_0} \times \mathcal{W}_{\t}^{n_0}$ for $n \ge n_0$.
To this end, it is enough to prove that $\mathcal{W}_{\t}^n \subset \Omega \times \mathcal{W}_{\t}^{n-1}$ for $n > n_0$.
In the following, we fix $n > n_0$.

Take any $\j=j_1 j_2 \ldots j_n \in \mathcal{A}_{\t}^n$.
Then there exist $\i=i_1 i_2 \ldots i_n \in \Omega_{\t}^n$ and $0\le j_0 \le m$ such that $\phi_{\j}(0) = \phi_{\i}(t_{j_0})$.
By Lemma \ref{lemma:i-1-equal-j-1} we have $j_1 = i_1$, and hence $\phi_{j_2 \ldots j_n}(0) = \phi_{i_2 \ldots i_n}(t_{j_0})$.
To verify that $j_2 \ldots j_n \in \mathcal{A}_{\t}^{n-1}$, it suffices to show that $i_2 \ldots i_n \in \Omega_{\t}^{n-1}$.
For any $1 \le k \le m$, since $\i \in \Omega_{\t}^n$, we have $\phi_{\i}(t_k) \in E_n$. Then there exists $\j'=j_1' j_2' \ldots j_n' \in \Omega^n$ (depending on $k$) such that $\phi_{\i}(t_{k}) = \phi_{\j'}(0)$. By Lemma \ref{lemma:i-1-equal-j-1} we have $i_1 = j_1'$, and hence $\phi_{i_2 \ldots i_n}(t_{k}) = \phi_{j_2' \ldots j_n'}(0)$. So we obtain that $\phi_{i_2 \ldots i_n}(t_{k}) \in E_{n-1}$ for all $1\le k \le m$, which implies that $i_2 \ldots i_n \in \Omega_{\t}^{n-1}$. Thus we have shown that $j_2 \ldots j_n \in \mathcal{A}_{\t}^{n-1}$.
Since $\j$ is arbitrarily chosen, we conclude that $\mathcal{A}_{\t}^n \subset \Omega \times \mathcal{A}_{\t}^{n-1}$.

We can show that $\widehat{\mathcal{A}}_{\t}^n \subset \Omega \times \widehat{\mathcal{A}}_{\t}^{n-1}$ using almost identical arguments.
Take any $\j=j_1 j_2 \ldots j_n \in \widehat{\mathcal{A}}_{\t}^n$.
Then there exist $\i=i_1 i_2 \ldots i_n \in \widehat{\Omega}_{\t}^n$ and $0\le j_0 \le m$ such that $\phi_{\j}(0) = \phi_{\i}(-t_{j_0})$, i.e., $\phi_{\j}(t_{j_0}) = \phi_{\i}(0)$.
By Lemma \ref{lemma:i-1-equal-j-1} we have $j_1 = i_1$, and hence $\phi_{j_2 \ldots j_n}(t_{j_0}) = \phi_{i_2 \ldots i_n}(0)$. That is, $\phi_{j_2 \ldots j_n}(0) = \phi_{i_2 \ldots i_n}(-t_{j_0})$.
To verify that $j_2 \ldots j_n \in \widehat{\mathcal{A}}_{\t}^{n-1}$, it is enough to show that $i_2 \ldots i_n \in \widehat{\Omega}_{\t}^{n-1}$.
For any $1 \le k \le m$, since $\i \in \widehat{\Omega}_{\t}^n$, we have $\phi_{\i}(-t_k) \in E_n$. Then there exists $\j'=j_1' j_2' \ldots j_n' \in \Omega^n$ (depending on $k$) such that $\phi_{\i}(-t_{k}) = \phi_{\j'}(0)$, i.e., $\phi_{\i}(0) = \phi_{\j'}(t_k)$. By Lemma \ref{lemma:i-1-equal-j-1} we have $i_1 = j_1'$, and hence $\phi_{i_2 \ldots i_n}(0) = \phi_{j_2' \ldots j_n'}(t_k)$. That is, $\phi_{i_2 \ldots i_n}(-t_k) = \phi_{j_2' \ldots j_n'}(0)$. So we obtain that $\phi_{i_2 \ldots i_n}(-t_{k}) \in E_{n-1}$ for all $1\le k \le m$, which implies that $i_2 \ldots i_n \in \widehat{\Omega}_{\t}^{n-1}$. Thus we have verified that $j_2 \ldots j_n \in \widehat{\mathcal{A}}_{\t}^{n-1}$.
Since $\j$ is arbitrary, we conclude that $\widehat{\mathcal{A}}_{\t}^n \subset \Omega \times \widehat{\mathcal{A}}_{\t}^{n-1}$.

Note that $\mathcal{W}_{\t}^n = \mathcal{A}_{\t}^n \cup \widehat{\mathcal{A}}_{\t}^n$.
It follows that $\mathcal{W}_{\t}^n \subset \Omega \times \mathcal{W}_{\t}^{n-1}$.
The proof is complete.
\end{proof}

Recall the construction of the directed graph $G_{\t}=(V_{\t}, E_{\t})$, where
$V_{\t} = \Omega^{\tau_{\t}+ \vartheta(\beta)} \setminus W_{\t}^{\tau_{\t}+ \vartheta(\beta)}$, and for $\i=i_1 i_2 \ldots i_{\tau_{\t}+ \vartheta(\beta)},\; \j=j_1 j_2 \ldots j_{\tau_{\t}+ \vartheta(\beta)} \in V_{\t}$, the edge from $\i$ to $\j$ belongs to $E_{\t}$ if and only if $i_2 \ldots i_{\tau_{\t}+ \vartheta(\beta)} = j_1 \ldots j_{\tau_{\t}+ \vartheta(\beta) -1}$.

\begin{proof}[Proof of Theorem \ref{thm:graph-criterion}]
  Write $n_0 = \tau_{\t}+\vartheta(\beta)$. 
  Note that $\t \in T^{m+1}$. 
  We first claim that $\t$ is an admissible translation vector if and only if there exists $\ell \ge n_0$ such that \[\bigcup_{n=n_0}^\ell \mathcal{W}_{\t}^n \times \Omega^{\ell-n} = \Omega^\ell. \]
  By Proposition \ref{prop:symbol-criterion}, the sufficiency is obvious.
  Next, suppose that $\t$ is an admissible translation vector. Then by Proposition \ref{prop:symbol-criterion} there exists $\ell \ge \tau_{\t}$ such that
  \[ \bigcup_{n=\tau_{\t}}^\ell \mathcal{W}_{\t}^n \times \Omega^{\ell-n} = \Omega^\ell. \]
  Let $\ell'= \ell+\vartheta(\beta)$.
  By Lemma \ref{lemma:W-subset}, we have
  \[ \bigcup_{n=n_0}^{\ell'} \mathcal{W}_{\t}^n \times \Omega^{\ell'-n} = \bigcup_{n=\tau_{\t}}^\ell \mathcal{W}_{\t}^{n+\vartheta(\beta)} \times \Omega^{\ell-n} \supset \bigcup_{n=\tau_{\t}}^\ell \Omega^{\vartheta(\beta)} \times \mathcal{W}_{\t}^{n} \times \Omega^{\ell-n} = \Omega^{\vartheta(\beta)} \times \Omega^{\ell} = \Omega^{\ell'}. \]
  The inverse inclusion is clear.
  Thus, we have proved the claim.

  By the claim and Proposition \ref{prop:W-equality}, $\t$ is an admissible translation vector if and only if there exists $\ell \ge n_0$ such that
  \begin{equation}\label{eq:symbol-criterion-2}
    \bigcup_{n=n_0}^\ell \Omega^{n-n_0} \times \mathcal{W}_{\t}^{n_0} \times \Omega^{\ell-n} = \Omega^\ell.
  \end{equation}
  It remains to show that (\ref{eq:symbol-criterion-2}) holds if and only if the directed graph $G_{\t}$ contains no cycles, which has already been proved in \cite[Proposition 1.3]{KLWYZ-2024}. For completeness, we adapt the proof of \cite[Proposition 1.3]{KLWYZ-2024}.

  Suppose that (\ref{eq:symbol-criterion-2}) does not hold for any $\ell \ge n_0$. 
  Choose a sufficiently large integer $\ell$. 
  We can find a finite word $i_1 i_2 \ldots i_\ell \in \Omega^{\ell}$ such that $i_1 i_2 \ldots i_\ell \notin \bigcup_{n=n_0}^\ell \Omega^{n-n_0} \times \mathcal{W}_{\t}^{n_0} \times \Omega^{\ell-n}$.
  This means that $i_{k} i_{k+1} \ldots i_{k+n_0 - 1} \notin \mathcal{W}_{\t}^{n_0}$ for any $1\le k \le \ell- n_0+1$. That is, $i_{k} i_{k+1} \ldots i_{k+n_0-1} \in V_{\t}$ for $1\le k \le \ell - n_0+1$.
  So the directed graph $G_{\t}$ contains the path 
  \begin{equation}\label{eq:path}
    i_1 i_2 \ldots i_{n_0} \longrightarrow i_2 i_3 \ldots i_{n_0+1} \longrightarrow i_3 i_4 \ldots i_{n_0+2} \longrightarrow \cdots\cdots \longrightarrow i_{\ell-n_0 + 1} i_{\ell-n_0 + 2} \ldots  i_\ell
  \end{equation}
  with length $\ell - n_0$. 
  The integer $\ell$ can be chosen such that $\ell-n_0 > (N+1)^{n_0}$, which guarantees that some vertex appears at least twice in the path in (\ref{eq:path}).
  Thus, the directed graph $G_{\t}$ contains a cycle.

  Next, suppose that the directed graph $G_{\t}$ contains a cycle.
  This implies that the directed graph $G_{\t}$ contains paths with arbitrary length.
  For any $\ell \ge n_0$, we can find a finite word $i_1 i_2 \ldots i_\ell \in \Omega^{\ell}$ such that $i_{k} i_{k+1} \ldots i_{k+n_0-1} \in V_{\t} = \Omega^{n_0} \setminus \mathcal{W}_{\t}^{n_0}$ for $1\le k \le \ell - n_0+1$.
  This means that $i_1 i_2 \ldots i_\ell \notin \bigcup_{n=n_0}^\ell \Omega^{n-n_0} \times \mathcal{W}_{\t}^{n_0} \times \Omega^{\ell-n}$.
  Thus, we obtain that (\ref{eq:symbol-criterion-2}) does not hold for any $\ell \ge n_0$.
  The proof is complete.
\end{proof}

\section{Proof of Corollary \ref{cor:m=1} and Examples}\label{sec:4}

In this section, we first apply Theorems \ref{thm:union-self-similar} and \ref{thm:graph-criterion} to prove Corollary \ref{cor:m=1}.
Recall that $\mathcal{B}_N$ is the set of $0< \beta < 1/(N+1)$ such that for any $n \in \N$, any $i_1,\ldots, i_n \in \{0,1,\ldots, 2N\}$, and any $j_1,\ldots, j_n \in \{0,1,\ldots,N\}$,
\[ \text{the equality}\;\;\sum_{k=1}^{n} i_k \beta^k = \sum_{k=1}^{n} j_k \beta^k \;\;\text{implies}\;\; i_k = j_k \;\;\forall 1 \le k \le n. \]

\begin{proof}[Proof of Corollary \ref{cor:m=1}]
Note first that $\beta \in \mathcal{B}_N$.
Write $\t=(t_0,t_1)$ with $0=t_0<t_1 = t$. Then we have $\Ga_{\t} = \Ga \cup (\Ga+t)$.
Note that $\widehat{\t} = \t$.
By Theorem \ref{thm:union-self-similar}, $\Ga_{\t}$ is a self-similar set if and only if $\t$ is an admissible translation vector.

In the following, we assume that $t \in T$ and write $\tau=\tau_{\t}$.
Recall that $\tau\in \N$ is the smallest integer such that $t \in T_{\tau}$.
There exists a unique finite word $\j = j_1 j_2 \ldots j_{\tau} \in \Omega^{\tau}$ such that $t = \beta^{-\tau}\phi_{\j}(0)$.
Note by the minimality of $\tau$ that $j_1 \ne 0$.
For $\i= i_1 i_2 \ldots i_{\tau} \in \Omega^{\tau}$, if $\phi_{\i}(t) = \phi_{\i}(0) + \phi_{\j}(0) \in E_{\tau}$, then noting that $\beta \in \mathcal{B}_N$ we have $(i_1 +j_1)(i_2+j_2) \ldots (i_{\tau} + j_{\tau}) \in \Omega^{\tau}$.
Thus, we have
\[\Omega_{\t}^{\tau} =\{ \i \in \Omega^{\tau}: \phi_{\i}(t) \in E_{\tau} \big\} = \big\{ i_1 i_2 \ldots i_{\tau} \in \Omega^{\tau}: i_{k} \le N- j_k \;\;\forall 1 \le k \le \tau \}.\]
It follows that $\mathcal{A}_{\t}^{\tau} = \Omega_{\t}^{\tau_{\t}} \cup \{ (i_1 +j_1)(i_2+j_2) \ldots (i_{\tau} + j_{\tau}): i_1 i_2 \ldots i_{\tau} \in \Omega_{\t}^{\tau} \}$. That is,
\[ \mathcal{A}_{\t}^{\tau} = \{ i_1 \ldots i_{\tau} \in \Omega^{\tau}: i_{k} \le N- j_k \;\forall 1 \le k \le \tau \} \cup \{ i_1 \ldots i_{\tau} \in \Omega^{\tau}:  i_{k} \ge j_k \;\forall 1 \le k \le \tau \}. \]
In the same way, we can obtain that
\[ \widehat{\Omega}_{\t}^{\tau} =\{ \i \in \Omega^{\tau}: \phi_{\i}(-t) \in E_{\tau} \big\} = \big\{ i_1 i_2 \ldots i_{\tau} \in \Omega^{\tau}:  i_{k} \ge j_k \;\;\forall 1 \le k \le \tau \},\]
and
\[ \widehat{\mathcal{A}}_{\t}^{\tau} = \widehat{\Omega}_{\t}^{\tau} \cup \{ (i_1 -j_1)(i_2-j_2) \ldots (i_{\tau} - j_{\tau}): i_1 i_2 \ldots i_{\tau} \in \widehat{\Omega}_{\t}^{\tau} \} = \mathcal{A}_{\t}^{\tau}.  \]
Thus, we conclude that
\begin{equation}\label{eq:W-t-m=1}
  \begin{split}
     \mathcal{W}_{\t}^{\tau} &= \{ i_1 i_2 \ldots i_{\tau} \in \Omega^{\tau}: i_{k} \le N- j_k \;\forall 1 \le k \le \tau \} \\
     & \quad\quad\quad\cup \{ i_1 i_2\ldots i_{\tau} \in \Omega^{\tau}:  i_{k} \ge j_k \;\forall 1 \le k \le \tau \}.
  \end{split}
\end{equation}
Since $\beta \in \mathcal{B}_N$, we have $\vartheta(\beta) =0$.
Recall the construction of the directed graph $G_{\t}=(V_{\t}, E_{\t})$ where $V_{\t} = \Omega^{\tau} \setminus \mathcal{W}_{\t}^{\tau}$.
By Theorem \ref{thm:graph-criterion}, $\t$ is an admissible translation vector if and only if the directed graph $G_{\t}$ contains no cycles.
The following arguments are adapted from the proof of \cite[Corollary 1.6]{KLWYZ-2024}.

\textbf{Case (i)}: $\tau=1$, or $\tau\ge 2$ and $j_k =0$ for $2\le k \le \tau$. Then by (\ref{eq:W-t-m=1}) we have
\[ \mathcal{W}_{\t}^{\tau} = \big( \{0,1,\ldots, N-j_{1}\} \cup \{ j_{1}, j_{1}+1, \ldots, N \} \big) \times \Omega^{\tau-1}. \]
In this case, the directed graph $G_{\t}$ contains no cycles if and only if
\[ \{0,1,\ldots, N-j_{1}\} \cup \{ j_{1}, j_{1}+1, \ldots, N \} = \Omega, \]
which is equivalent to $j_1 \le N - j_1 + 1$, i.e., $j_1\le (N+1)/2$.

\textbf{Case (ii)}: $\tau\ge 2$ and there exists $2\le q \le \tau$ such that $j_{q} \ne 0$.
For $\i=i_1 i_2 \ldots i_\tau \in \Omega^{\tau}$, if $\i \in \mathcal{W}_{\t}^{\tau}$ then by (\ref{eq:W-t-m=1}) we have either $i_1\le N- j_1,\; i_q \le N - j_q$ or $i_1\ge j_1,\; i_q \ge j_q$. Note that $j_1 \ne 0$ and $j_q \ne 0$.
Thus, for $\i=i_1 i_2 \ldots i_\tau \in \Omega^{\tau}$, if $(i_1, i_q) = (0,N)$ or $(i_1, i_q) =(N,0)$, then $\i \not\in \mathcal{W}_{\t}^{\tau}$.
Let $(i_k)_{k=1}^\f = 0^{q-1}N^{q-1} 0^{q-1}N^{q-1} \ldots$.
Then for any $k \in \N$, we have $i_k i_{k+1} \ldots i_{k+\tau-1} \notin \mathcal{W}_{\t}^{\tau}$, i.e., $i_k i_{k+1} \ldots i_{k+\tau-1} \in E_{\t}$.
Thus, the directed graph $G_{\t}$ contains the following cycle
\[ i_1 i_2 \ldots i_{\tau} \longrightarrow i_2 i_3 \ldots i_{\tau+1} \longrightarrow i_3 i_4 \ldots i_{\tau+2} \longrightarrow \cdots\cdots \longrightarrow i_{2q-1} i_{2q} \ldots i_{2q+\tau-2}. \]

Therefore, we conclude that for $t >0$, the union $\Ga \cup (\Ga+t)$ is a self-similar set if and only if \[ t=\frac{j_1(1-\beta)}{N} \beta^{-\tau}, \]
where $j_1 \in \big\{ 1, 2, \ldots, \lfloor\frac{N+1}{2}\rfloor \big\}$ and $\tau\in\N$.
\end{proof}

Next, we apply Theorems \ref{thm:union-self-similar} and \ref{thm:graph-criterion} to provide two examples to illustrate the self-similarity of the union $\Ga \cup (\Ga + t)$.

\begin{example}\label{example:not-in-B}
  Fix $N \in \N$ and let $\beta \in \big( 0, 1/(N+1) \big)$ satisfy
  \begin{equation}\label{eq:example-beta}
    (N+1) \beta + N \beta^2 = 1.
  \end{equation}
  Note that $\beta \notin \mathcal{B}_N$ and $\vartheta(\beta)=1$. Let $\Ga$ be the attractor of the self-similar IFS in (\ref{eq:IFS-Ga}).
  Let \[ t = \beta^{-2} \phi_{NN}(0) = \frac{1-\beta^2}{\beta^2}.\]
  Then the union $\Ga \cup (\Ga+t)$ is a self-similar set.
  By Corollary \ref{cor:scale}, the union $\Ga \cup (\Ga+\beta^{-k}t)$ is also a self-similar set for any $k\in \N$.

  Write $\t=(0,t)$. Then we have $\tau_{\t}=2$.
  For $\i \in \Omega^3$, we have $\phi_{\i}(t) = \phi_{\i}(0) + \phi_{0NN}(0)$ and $\phi_{\i}(-t) = \phi_{\i}(0) - \phi_{0NN}(0)$.
  Using (\ref{eq:example-beta}), we obtain that
  \begin{align*}
  \Omega_{\t}^{3} & = \{ i_1 00: 0 \le i_1 \le N\} \cup \{ i_1 i_2 i_3: \; 0\le i_1 \le N-1,\; 1\le i_2 \le N,\; 0\le i_3 \le N  \}, \\
  \widehat{\Omega}_{\t}^{3} & = \{ i_1NN: 0 \le i_1 \le N\} \cup \{ i_1 i_2 i_3: \; 1\le i_1 \le N,\; 0\le i_2 \le N-1,\; 0\le i_3 \le N  \}.
  \end{align*}
  It follows that
  \begin{align*}
    \mathcal{W}_{\t}^3 & = \{ i_1 00, i_1 NN: 0 \le i_1 \le N\} \cup \{ i_1 i_2 i_3: \; 0\le i_1 \le N-1,\; 1\le i_2 \le N,\; 0\le i_3 \le N  \}\\
    & \hspace{8em} \cup \{ i_1 i_2 i_3: \; 1\le i_1 \le N,\; 0\le i_2 \le N-1,\; 0\le i_3 \le N  \}.
  \end{align*}
  Thus, in the directed graph $G_{\t} = (V_{\t},E_{\t})$, we have
  \[ V_{\t}= \Omega^3 \setminus \mathcal{W}_{\t}^3 =\{ 00i_3: 1\le i_3 \le N \} \cup \{ NN i_3: 0\le i_3 \le N-1 \}. \]
  One can check that $E_{\t}=\emptyset$.
  This means that the directed graph $G_{\t}$ contains no cycles.
  By Theorems \ref{thm:union-self-similar} and \ref{thm:graph-criterion}, the union $\Ga \cup (\Ga+t)$ is a self-similar set.
\end{example}

\begin{example}\label{example-2:not-in-B}
  Fix an integer $N \ge 2$ and let $\beta \in \big( 0, 1/(N+1) \big)$ satisfy
  \begin{equation}\label{eq:example-beta-2}
    (N+1) \beta + \beta^2 = 1.
  \end{equation}
  Note that $\beta \notin \mathcal{B}_N$ and $\vartheta(\beta)=2$. Let $\Ga$ be the attractor of the self-similar IFS in (\ref{eq:IFS-Ga}).
  Let \[ t = \beta^{-1} \phi_{N}(0) = \frac{1-\beta}{\beta} .\]
  Then the union $\Ga \cup (\Ga+t)$ is not a self-similar set, but the union $\Ga \cup (\Ga+ \beta^{-k}t)$ is a self-similar set for any $k \in \N$.

  Write $\t=(0,t)$. Then we have $\tau_{\t}=1$.
  For $\i \in \Omega^3$, we have $\phi_{\i}(t) = \phi_{\i}(0) + \phi_{00N}(0)$ and $\phi_{\i}(-t) = \phi_{\i}(0) - \phi_{00N}(0)$.
  It is easy to check that $\Omega_{\t}^{3} = \Omega^2 \times \{0\}$ and $\widehat{\Omega}_{\t}^{3} = \Omega^2 \times \{N\}$.
  Thus, we obtain that $\mathcal{W}_{\t}^3 = \Omega^2 \times \{0,N\}$.
  Note that $N \ge 2$. In the directed graph $G_{\t}=(V_{\t},E_{\t})$, we have $111 \in V_{\t}=\Omega^3 \setminus \mathcal{W}_{\t}^3$. Thus, the directed graph $G_{\t}$ contains a cycle. By Theorems \ref{thm:union-self-similar} and \ref{thm:graph-criterion}, the union $\Ga \cup (\Ga+t)$ is not a self-similar set.

  Write $\t'=(0,\beta^{-1}t)$. Then we have $\tau_{\t'}=2$.
  For $\i \in \Omega^4$, we have $\phi_{\i}(\beta^{-1}t) = \phi_{\i}(0) + \phi_{00N0}(0)$ and $\phi_{\i}(-t) = \phi_{\i}(0) - \phi_{00N0}(0)$.
  Using (\ref{eq:example-beta-2}), we obtain that
  \begin{align*}
  \Omega_{\t'}^{4} & = \{ i_1 i_2 0 i_4: 0 \le i_1,i_2,i_4 \le N\} \cup \{ i_1 N i_3 i_4:  0\le i_1 \le N-1,\; 2\le i_3 \le N,\; 1\le i_4 \le N  \}  \\
  & \hspace{6em} \cup \{ i_1 i_2 i_3 i_4: 0\le i_1 \le N,\; 0\le i_2 \le N-1,\; 1\le i_3 \le N,\; 1\le i_4 \le N  \}, \\
  \widehat{\Omega}_{\t'}^{4} & = \{ i_1 i_2 N i_4: 0 \le i_1,i_2,i_4 \le N\} \cup \{ i_1 0 i_3 i_4: 1\le i_1 \le N,\; 0\le i_3 \le N-2,\; 0\le i_4 \le N-1  \} \\
  & \hspace{5em} \cup \{ i_1 i_2 i_3 i_4: \; 0\le i_1 \le N,\; 1\le i_2 \le N,\; 0\le i_3 \le N-1,\; 0\le i_4 \le N-1  \}.
  \end{align*}
  It follows that $\mathcal{W}_{\t'}^4 = \Omega_{\t'}^{4} \cup \widehat{\Omega}_{\t'}^{4}$.
  In the directed graph $G_{\t'}=(V_{\t'},E_{\t'})$, we have
  \[ V_{\t'} = \Omega^4 \setminus \mathcal{W}_{\t'}^4 = \{ N N i_3 N, 0 0 i_3 0: 1\le i_3 \le N-1 \} \cup \{ i_1 N 1 N, i_10(N-1)0: 0 \le i_1 \le N \}. \]
  One can check that $E_{\t'} =\emptyset$.
  Thus, the directed graph $G_{\t'}$ contains no cycles. By Theorems \ref{thm:union-self-similar} and \ref{thm:graph-criterion}, the union $\Ga \cup (\Ga+\beta^{-1}t)$ is a self-similar set.
  By Corollary \ref{cor:scale}, the union $\Ga \cup (\Ga+\beta^{-k}t)$ is a self-similar set for any $k\in \N$.
\end{example}

We conclude this section with some remarks.
As shown in Examples \ref{example:not-in-B} and \ref{example-2:not-in-B}, the self-similarity of the union $\Ga \cup (\Ga +t)$ becomes intricate when $\beta \notin \mathcal{B}_N$, primarily depending on the algebraic equation satisfied by $\beta$.
As in \cite[Lemma 4.3]{KLWYZ-2024}, we can show that for $\beta \in \mathcal{B}_N$ and for $\t \in T^{m+1}$, $\t$ is an admissible translation vector if and only if $\beta^{-k} \t$ is an admissible translation vector for any $k \in \N$.
However, by Example \ref{example-2:not-in-B}, this necessity no longer holds when $\beta \not\in \mathcal{B}_N$.

\section{Proof of Theorems \ref{thm:translation-countable} and \ref{thm:not-self-similar}}\label{sec:5}

In this section, we will prove Theorems \ref{thm:translation-countable} and \ref{thm:not-self-similar}.
For a similitude $f$ on $\R$, let $\rho_f$ denote its linear coefficient.

\begin{proof}[Proof of Theorem \ref{thm:translation-countable}]
  Recall that $K \subset \R$ is a self-similar set that satisfies the SSC.
  By \cite[Proposition 4.3 (i)]{Elekes-Keleti-Mathe-2010}, the set
  \[ \mathcal{S} = \big\{ f: \; f\; \text{is a similitude on}\; \R,\;\text{and}\; f(K) \subset K \big\} \]
  is discrete as a subset of $\R^2$ and, hence, is countable.
  It follows that the set \[ \mathcal{T} = \big\{ \rho^{-1} \big( f_1(0) - f_2(0) \big): f_1, f_2 \in \mathcal{S},\; \rho = \rho_{f_1} = \rho_{f_2} \big\} \]
  is at most countable.

  Write $\t=(t_0,t_1, \ldots, t_m) \in \R^{m+1}$, where $0=t_0< t_1 < \cdots < t_m$.
  Suppose that the union $K_{\t} = \bigcup_{j=0}^m (K+t_j)$ is a self-similar set.
  By Proposition \ref{prop:union-similarity-general}, there exists a finite set $\mathcal G$ of similitudes such that \[ \bigcup_{g\in\mathcal G} g(K_{\t})=K. \]
  Note that $K \subset K_{\t}$. We conclude that $\mathcal{G} \subset \mathcal{S}$.
  Choose arbitrarily $g \in \mathcal{G}$ and fix $1 \le j \le m$.
  Note that $g(K+t_j) \subset g(K_{\t}) \subset K$.
  Then there exists $\widetilde{g} \in \mathcal{G}$ such that $\rho_{\widetilde{g}} = \rho_{g}$ and $g(x+t_j) = \widetilde{g}(x)$.
  It follows that \[ t_j = \rho_{g}^{-1}\big( \widetilde{g}(0) - g(0) \big) \in \mathcal{T}. \]
  Thus, we obtain that $(t_1,\ldots,t_m) \in \mathcal{T}^m$, which is at most countable.
  We complete the proof.
\end{proof}

The proof of Theorem \ref{thm:not-self-similar} is divided into two parts.

\begin{proof}[Proof of Theorem \ref{thm:not-self-similar} {\rm(i)}]
  Recall that $K \subset \R$ is the attractor of the self-similar IFS $\big\{ \varphi_{0}(x) = \lambda_0 x,\; \varphi_{1}(x) = \lambda_1 x + 1- \lambda_1 \big\}$, where $\lambda_0,\lambda_1>0$, $\lambda_0 + \lambda_1 < 1$, and $\log \lambda_0/ \log \lambda_1 \not\in \Q$.
  For $\i = i_1 i_2 \ldots i_n \in \bigcup_{k=1}^\f\{0,1\}^k$, write $\varphi_{\i} = \varphi_{i_1} \circ \varphi_{i_2} \circ \cdots \circ \varphi_{i_n}$.
  Clearly, we have $\lambda_0 \ne \lambda_1$.
  By \cite[Corollary 1.3 (i)]{Wang-2025}, every contractive similitude $f$ on $\R$ satisfying $f(K) \subset K$ has the form $f(x)=\varphi_{\i}(x)$ for some $\i \in \bigcup_{k=1}^\f\{0,1\}^k$.
  Let $\t=(t_0,t_1,\ldots, t_m)\in\R^{m+1}$ with $0=t_0<t_1<\cdots<t_m$

  Suppose on the contrary that the union $K_{\t}=\bigcup_{j=0}^m (K+t_j)$ is a self-similar set.
  By Proposition \ref{prop:union-similarity-general}, there exists a finite set $\mathcal{G}$ of similitudes such that
  \begin{equation}\label{eq:G-K-2}
    \bigcup_{ g \in \mathcal{G}} g(K_{\t}) = K.
  \end{equation}
  Note that $K \subset K_{\t}$. Each $g \in \mathcal{G}$ has the form $g(x)=\varphi_{\i}(x)$ for some $\i \in \bigcup_{k=1}^\f\{0,1\}^k$.
  Note that $0 \in K$. By (\ref{eq:G-K-2}), there exists $\i \in \bigcup_{k=1}^\f\{0,1\}^k$ such that $0 \in \varphi_{\i}(K_{\t}) \subset K$.
  Note also that $0$ is the minimum of $K$. Thus, we obtain that $\varphi_{\i}(0) = 0$, which implies that $\i=0^n$, where $n \in \N$ is the length of $\i$.
  Note that $\varphi_{\i}(K+t_m) \subset \varphi_{\i}(K_{\t}) \subset K$.
  Then there exists $\j \in \bigcup_{k=1}^\f\{0,1\}^k$ such that $\varphi_{\i}(x+t_m) = \varphi_{\j}(x)$. This implies that $\rho_{\varphi_{\j}} = \rho_{\varphi_{\i}}= \lambda_0^n$.
  Since $\log \lambda_0/ \log \lambda_1 \not\in \Q$, we conclude that $\j = \i =0^n$.
  So we obtain $t_m=0$, a contradiction.
  Therefore, we conclude that the union $K_{\t}$ is not a self-similar set.
\end{proof}

\begin{proof}[Proof of Theorem \ref{thm:not-self-similar} {\rm(ii)}]
  Recall that $K \subset \R$ is the attractor of the self-similar IFS $\big\{ \varphi_{0}(x) = \lambda x,\; \varphi_{1}(x) = \lambda x + \zeta,\;\varphi_2(x) = \lambda x + 1- \lambda \big\}$, where $0< \lambda < 1/3$, $\lambda < \zeta < 1-2\lambda$, and $\zeta\not\in \Q(\lambda)$.
  For $\i = i_1 i_2 \ldots i_n \in \bigcup_{k=1}^\f\{0,1,2\}^k$, let $\Lambda_{\i} = \{ 1 \le k \le n: i_k =1\}$ and $\varphi_{\i} = \varphi_{i_1} \circ \varphi_{i_2} \circ \cdots \circ \varphi_{i_n}$. Then we have
  \[ \varphi_{\i}(0) = \frac{1-\lambda}{2} \sum_{k \in \{1,2,\ldots,n\} \setminus\Lambda_{\i}} i_k\lambda^{k-1} + \zeta \sum_{k \in \Lambda_{\i}} \lambda^{k-1}. \]
  Clearly, we have $\zeta \ne (1-\lambda)/2$.
  By \cite[Theorem 1.1 (i)]{Wang-2025}, every contractive similitude $f$ satisfying $f(K) \subset K$ has the form $f(x)=\varphi_{\i}(x)$ foe some $\i \in \bigcup_{k=1}^\f\{0,1,2\}^k$.

  For $t >0$, suppose on the contrary that the union $K \cup (K+t)$ is a self-similar set.
  By Proposition \ref{prop:union-similarity-general}, there exists a finite set $\mathcal{G}$ of similitudes such that
  \begin{equation}\label{eq:G-K-3}
    \bigcup_{ g \in \mathcal{G}} g\big( K \cup (K+t) \big) = K.
  \end{equation}
  Note that $g(K) \subset K$ for $g \in \mathcal{G}$. So each $g \in \mathcal{G}$ has the form $g(x)=\varphi_{\i}(x)$ for some $\i \in \bigcup_{k=1}^\f\{0,1,2\}^k$.
  Let $x_0 = \zeta/(1-\lambda)$, which is the fixed point of $\varphi_1$. 
  By (\ref{eq:G-K-3}), there exists $\i \in \bigcup_{k=1}^\f\{0,1,2\}^k$ such that $x_0 \in \varphi_{\i}\big( K \cup (K+t) \big) \subset K$.
  There are only two possible cases: (i) $x_0 \in \varphi_{\i}(K)$; (ii) $x_0 \in \varphi_{\i}(K+t)$.

  \textbf{Case (i)}: $x_0 \in \varphi_{\i}(K)$. 
  Note that $x_0 \in \varphi_{1^k}(K)$ for all $k \in \N$.
  We have $\i = 1^n$, where $n \in \N$ is the length of $\i$.
  Note that $\varphi_{\i}(K+t) \subset K$.
  There exists $\j = j_1 j_2 \ldots j_n \in \{0,1,2\}^n$ such that $\varphi_{\i}(x+t) = \varphi_{\j}(x)$.
  It follows that $\j \ne \i$, and
  \begin{equation}\label{eq:t}
    t = \lambda^{-n}\big( \varphi_{\j}(0) - \varphi_{\i}(0) \big) = \lambda^{-n} \bigg( \frac{1-\lambda}{2} \sum_{k \in \{1,2,\ldots,n\} \setminus\Lambda_{\j}} j_k\lambda^{k-1} - \zeta \sum_{k \in \{1,2,\ldots,n\} \setminus\Lambda_{\j}} \lambda^{k-1} \bigg).
  \end{equation}
  Since $0 \in K$, by (\ref{eq:G-K-3}) there exists $\i' \in \bigcup_{k=1}^\f\{0,1,2\}^k$ such that $0 \in \varphi_{\i'}\big( K \cup (K+t) \big) \subset K$.
  Note that $0$ is the minimum of $K$. Thus, we obtain that $\varphi_{\i'}(0) = 0$, which implies that $\i'=0^\ell$, where $\ell \in \N$ is the length of $\i'$.
  Note that $\varphi_{\i'}(K+t) \subset K$. Then there exists $\j'=j_1' j_2' \ldots j_\ell' \in \{0,1,2\}^\ell$ such that $\varphi_{\i'}(x+t) = \varphi_{\j'}(x)$. It follows that
  \[ t=\lambda^{-\ell} \varphi_{\j'}(0) = \lambda^{-\ell} \bigg( \frac{1-\lambda}{2} \sum_{k \in \{1,2,\ldots,\ell\} \setminus\Lambda_{\j'}} j_k'\lambda^{k-1} + \zeta \sum_{k \in \Lambda_{\j'}} \lambda^{k-1} \bigg).\]
  Together with (\ref{eq:t}), we obtain that
  \begin{align*}
    \zeta\bigg( \lambda^\ell \sum_{k \in \{1,2,\ldots,n\} \setminus\Lambda_{\j}} \lambda^{k-1} + \lambda^n \sum_{k \in \Lambda_{\j'}} \lambda^{k-1} \bigg)
    & = \frac{\lambda^\ell(1-\lambda)}{2} \sum_{k \in \{1,2,\ldots,n\} \setminus\Lambda_{\j}} j_k\lambda^{k-1} \\
    & \hspace{4em}- \frac{(1-\lambda)\lambda^n}{2} \sum_{k \in \{1,2,\ldots,\ell\} \setminus\Lambda_{\j'}} j_k'\lambda^{k-1}.
  \end{align*}
  Note that $\Lambda_{\j} \ne \{1,2,\ldots, n\}$.
  Thus, we conclude that $\zeta \in \Q(\lambda)$, a contradiction.

  \textbf{Case (ii)}: $x_0 \in \varphi_{\i}(K+t)$.
  Then there exists $\j \in \bigcup_{k=1}^\f\{0,1,2\}^k$ such that $\varphi_{\i}(x+t) = \varphi_{\j}(x)$.
  It follows that $x_0 \in \varphi_{\j}(K)$.
  This implies that $\j = 1^n$, where $n \in \N$ is the length of $\j$.
  Thus, we obtain that $\i = i_1 i_2 \ldots i_n \in \{0,1,2\}^n$, $\i \ne \j$, and
  \begin{equation}\label{eq:t-2}
    t = \lambda^{-n}\big( \varphi_{\j}(0) - \varphi_{\i}(0) \big) = \lambda^{-n} \bigg( \zeta \sum_{k \in \{1,2,\ldots,n\} \setminus\Lambda_{\i}} \lambda^{k-1}  - \frac{1-\lambda}{2} \sum_{k \in \{1,2,\ldots,n\} \setminus\Lambda_{\i}} i_k\lambda^{k-1}\bigg).
  \end{equation}
  Since $1 \in K$, by (\ref{eq:G-K-3}) there exists $\i' =i_1' i_2' \ldots i_\ell'\in \bigcup_{k=1}^\f\{0,1,2\}^k$ such that $1\in \varphi_{\i'}\big( K \cup (K+t) \big) \subset K$.
  Note that $1$ is the maximum of $K$. Thus, we obtain that $\varphi_{\i'}(1+t)=1$.
  It follows that \[ t=\lambda^{-\ell}\big( 1- \varphi_{\i'}(0)\big) -1 = \lambda^{-\ell}\bigg( 1-\frac{1-\lambda}{2} \sum_{k \in \{1,2,\ldots,\ell\} \setminus\Lambda_{\i'}} i_k'\lambda^{k-1} - \zeta \sum_{k \in \Lambda_{\i'}} \lambda^{k-1} \bigg) -1.\]
  Together with (\ref{eq:t-2}), we obtain that
  \begin{align*}
    \zeta\bigg( \lambda^\ell \sum_{k \in \{1,2,\ldots,n\} \setminus\Lambda_{\i}} \lambda^{k-1} + \lambda^n \sum_{k \in \Lambda_{\i'}} \lambda^{k-1} \bigg)
    & = \frac{\lambda^\ell(1-\lambda)}{2} \sum_{k \in \{1,2,\ldots,n\} \setminus\Lambda_{\i}} i_k\lambda^{k-1} +\lambda^n \\
    & \hspace{2em}- \frac{(1-\lambda)\lambda^n}{2} \sum_{k \in \{1,2,\ldots,\ell\} \setminus\Lambda_{\i'}} i_k'\lambda^{k-1} -\lambda^{n+\ell}.
  \end{align*}
  Note that $\Lambda_{\i} \ne \{1,2,\ldots, n\}$.
  Thus, we conclude that $\zeta \in \Q(\lambda)$, a contradiction.

  Therefore, the union $K \cup (K+t)$ is not a self-similar set for any $t >0$.
\end{proof}

\section*{Acknowledgements}
The author was supported by the National Natural Science Foundation of China (No. 12501110, 12471085) and the China Postdoctoral Science Foundation (No. 2024M763857).

\bibliographystyle{abbrv}
\bibliography{Self-Similarity}

@article {Hutchinson-1981,
    AUTHOR = {Hutchinson, John E.},
     TITLE = {Fractals and self-similarity},
   JOURNAL = {Indiana Univ. Math. J.},
  FJOURNAL = {Indiana University Mathematics Journal},
    VOLUME = {30},
      YEAR = {1981},
    NUMBER = {5},
     PAGES = {713--747},
      ISSN = {0022-2518,1943-5258},
   MRCLASS = {49F20 (00A69 28A12 58C27)},
  MRNUMBER = {625600},
MRREVIEWER = {F.\ J.\ Almgren, Jr.},
       DOI = {10.1512/iumj.1981.30.30055},
       URL = {https://doi.org/10.1512/iumj.1981.30.30055},
}

@article {KLWYZ-2024,
    AUTHOR = {Kong, Derong and Li, Wenxia and Wang, Zhiqiang and Yao,
              Yuanyuan and Zhang, Yunxiu},
     TITLE = {On the union of homogeneous symmetric {C}antor set with its
              translations},
   JOURNAL = {Math. Z.},
  FJOURNAL = {Mathematische Zeitschrift},
    VOLUME = {307},
      YEAR = {2024},
    NUMBER = {2},
     PAGES = {Paper No. 35, 20},
      ISSN = {0025-5874,1432-1823},
   MRCLASS = {28A80 (28A78)},
  MRNUMBER = {4747087},
MRREVIEWER = {Filip\ Grzegorz\ Strobin},
       DOI = {10.1007/s00209-024-03499-4},
       URL = {https://doi.org/10.1007/s00209-024-03499-4},
}

@article {Wang-2025,
    AUTHOR = {Wang, Zhiqiang},
     TITLE = {Self-embeddings of homogeneous self-similar sets generated by three maps},
   JOURNAL = {arXiv:2509.18202},
  FJOURNAL = {arXiv:2509.18202},
      YEAR = {2025},
       URL = {https://arxiv.org/abs/arXiv:2509.18202},
}

@article {Deng-Liu-2011,
    AUTHOR = {Deng, Guo-Tai and Liu, Chun-Tai},
     TITLE = {Self-similarity of unions of {C}antor sets},
   JOURNAL = {J. Math. (Wuhan)},
  FJOURNAL = {Journal of Mathematics. Shuxue Zazhi},
    VOLUME = {31},
      YEAR = {2011},
    NUMBER = {5},
     PAGES = {847--852},
      ISSN = {0255-7797},
   MRCLASS = {28A80},
  MRNUMBER = {2894996},
MRREVIEWER = {Andreas\ Schief},
}

@article {Elekes-Keleti-Mathe-2010,
    AUTHOR = {Elekes, M\'arton and Keleti, Tam\'as and M\'ath\'e, Andr\'as},
     TITLE = {Self-similar and self-affine sets: measure of the intersection
              of two copies},
   JOURNAL = {Ergodic Theory Dynam. Systems},
  FJOURNAL = {Ergodic Theory and Dynamical Systems},
    VOLUME = {30},
      YEAR = {2010},
    NUMBER = {2},
     PAGES = {399--440},
      ISSN = {0143-3857,1469-4417},
   MRCLASS = {28A80 (37C45)},
  MRNUMBER = {2599886},
MRREVIEWER = {Antti\ K\"aenm\"aki},
       DOI = {10.1017/S0143385709000121},
       URL = {https://doi.org/10.1017/S0143385709000121},
}

@article {Moreira-Xi-Zhang-2025,
    AUTHOR = {Moreira, Carlos Gustavo and Xi, Jinghua and Zhang, Yiwei},
     TITLE = {Graphs of continuous but non-affine functions are never
              self-similar},
   JOURNAL = {Proc. Amer. Math. Soc.},
  FJOURNAL = {Proceedings of the American Mathematical Society},
    VOLUME = {153},
      YEAR = {2025},
    NUMBER = {6},
     PAGES = {2501--2512},
      ISSN = {0002-9939,1088-6826},
   MRCLASS = {28A80},
  MRNUMBER = {4892623},
MRREVIEWER = {Dah-Chin\ Luor},
       DOI = {10.1090/proc/17099},
       URL = {https://doi.org/10.1090/proc/17099},
}

@article {Bandt-Kravchenko-2011,
    AUTHOR = {Bandt, Christoph and Kravchenko, Aleksey},
     TITLE = {Differentiability of fractal curves},
   JOURNAL = {Nonlinearity},
  FJOURNAL = {Nonlinearity},
    VOLUME = {24},
      YEAR = {2011},
    NUMBER = {10},
     PAGES = {2717--2728},
      ISSN = {0951-7715,1361-6544},
   MRCLASS = {28A80 (26A27)},
  MRNUMBER = {2834243},
MRREVIEWER = {Kiko\ Kawamura},
       DOI = {10.1088/0951-7715/24/10/003},
       URL = {https://doi.org/10.1088/0951-7715/24/10/003},
}

@article {Feng-Rao-Wang-2015,
    AUTHOR = {Feng, De-Jun and Rao, Hui and Wang, Yang},
     TITLE = {Self-similar subsets of the {C}antor set},
   JOURNAL = {Adv. Math.},
  FJOURNAL = {Advances in Mathematics},
    VOLUME = {281},
      YEAR = {2015},
     PAGES = {857--885},
      ISSN = {0001-8708,1090-2082},
   MRCLASS = {28A80 (11K16)},
  MRNUMBER = {3366855},
MRREVIEWER = {Eino\ Vihtori\ Rossi},
       DOI = {10.1016/j.aim.2015.06.002},
       URL = {https://doi.org/10.1016/j.aim.2015.06.002},
}

@article {Yao-Li-2012,
    AUTHOR = {Yao, Yuanyuan and Li, Wenxia},
     TITLE = {Self-similar structure on intersection of {C}artesian product
              of {C}antor triadic sets with their translations},
   JOURNAL = {Monatsh. Math.},
  FJOURNAL = {Monatshefte f\"ur Mathematik},
    VOLUME = {166},
      YEAR = {2012},
    NUMBER = {3-4},
     PAGES = {591--600},
      ISSN = {0026-9255,1436-5081},
   MRCLASS = {28A80},
  MRNUMBER = {2925157},
       DOI = {10.1007/s00605-011-0312-6},
       URL = {https://doi.org/10.1007/s00605-011-0312-6},
}

@article {Li-Yao-Zhang-2011,
    AUTHOR = {Li, Wenxia and Yao, Yuanyuan and Zhang, Yunxiu},
     TITLE = {Self-similar structure on intersection of homogeneous
              symmetric {C}antor sets},
   JOURNAL = {Math. Nachr.},
  FJOURNAL = {Mathematische Nachrichten},
    VOLUME = {284},
      YEAR = {2011},
    NUMBER = {2-3},
     PAGES = {298--316},
      ISSN = {0025-584X,1522-2616},
   MRCLASS = {28A80 (28A78)},
  MRNUMBER = {2790890},
MRREVIEWER = {Chaoshou\ Dai},
       DOI = {10.1002/mana.200710104},
       URL = {https://doi.org/10.1002/mana.200710104},
}

@article {Deng-He-Wen-2008,
    AUTHOR = {Deng, Guo-Tai and He, Xing-Gang and Wen, Zhi-Xiong},
     TITLE = {Self-similar structure on intersections of triadic {C}antor
              sets},
   JOURNAL = {J. Math. Anal. Appl.},
  FJOURNAL = {Journal of Mathematical Analysis and Applications},
    VOLUME = {337},
      YEAR = {2008},
    NUMBER = {1},
     PAGES = {617--631},
      ISSN = {0022-247X,1096-0813},
   MRCLASS = {28A80 (37C45)},
  MRNUMBER = {2356097},
MRREVIEWER = {David\ A.\ Croydon},
       DOI = {10.1016/j.jmaa.2007.03.089},
       URL = {https://doi.org/10.1016/j.jmaa.2007.03.089},
}

@article {Kong-Li-Dekking-2010,
    AUTHOR = {Kong, Derong and Li, Wenxia and Dekking, F. Michel},
     TITLE = {Intersections of homogeneous {C}antor sets and
              beta-expansions},
   JOURNAL = {Nonlinearity},
  FJOURNAL = {Nonlinearity},
    VOLUME = {23},
      YEAR = {2010},
    NUMBER = {11},
     PAGES = {2815--2834},
      ISSN = {0951-7715,1361-6544},
   MRCLASS = {28A80 (11A63 28A78)},
  MRNUMBER = {2727171},
       DOI = {10.1088/0951-7715/23/11/005},
       URL = {https://doi.org/10.1088/0951-7715/23/11/005},
}

@article {Zou-Lu-Li-2008,
    AUTHOR = {Zou, Yuru and Lu, Jian and Li, Wenxia},
     TITLE = {Self-similar structure on the intersection of
              middle-{$(1-2\beta)$} {C}antor sets with
              {$\beta\in(1/3,1/2)$}},
   JOURNAL = {Nonlinearity},
  FJOURNAL = {Nonlinearity},
    VOLUME = {21},
      YEAR = {2008},
    NUMBER = {12},
     PAGES = {2899--2910},
      ISSN = {0951-7715,1361-6544},
   MRCLASS = {28A80 (28A78)},
  MRNUMBER = {2461046},
       DOI = {10.1088/0951-7715/21/12/010},
       URL = {https://doi.org/10.1088/0951-7715/21/12/010},
}

@article {Pedersen-Phillips-2014,
    AUTHOR = {Pedersen, Steen and Phillips, Jason D.},
     TITLE = {On intersections of {C}antor sets: self-similarity},
   JOURNAL = {Commun. Math. Anal.},
  FJOURNAL = {Communications in Mathematical Analysis},
    VOLUME = {16},
      YEAR = {2014},
    NUMBER = {1},
     PAGES = {1--30},
      ISSN = {1938-9787},
   MRCLASS = {28A80},
  MRNUMBER = {3161733},
MRREVIEWER = {Xing-Gang\ He},
       URL = {http://projecteuclid.org/euclid.cma/1383587517},
}

@article{Xu-Xi-Jiang-2024,
  title = {On self-similarity of quadrangle},
  volume = {32},
  ISSN = {1793-6543},
  url = {http://dx.doi.org/10.1142/S0218348X24500968},
  DOI = {10.1142/s0218348x24500968},
  number = {05},
  PAGES = {2450096},
  journal = {Fractals},
  publisher = {World Scientific Pub Co Pte Ltd},
  author = {Xu,  Jiajun and Xi,  Jinghua and Jiang,  Kan},
  year = {2024},
}

@article{Wang-Jiang-Xi-2025,
  title = {ON SELF-SIMILARITY OF BOUNDED REGIONS IN THE PLANE},
  volume = {33},
  ISSN = {1793-6543},
  url = {http://dx.doi.org/10.1142/S0218348X25500537},
  DOI = {10.1142/s0218348x25500537},
  number = {07},
  PAGES = {2550053},
  journal = {Fractals},
  publisher = {World Scientific Pub Co Pte Ltd},
  author = {Wang, Xuemin and Jiang, Kan and Xi, Lifeng},
  year = {2025},
}

@article {Baker-Kong-2017,
    AUTHOR = {Baker, Simon and Kong, Derong},
     TITLE = {Unique expansions and intersections of {C}antor sets},
   JOURNAL = {Nonlinearity},
  FJOURNAL = {Nonlinearity},
    VOLUME = {30},
      YEAR = {2017},
    NUMBER = {4},
     PAGES = {1497--1512},
      ISSN = {0951-7715,1361-6544},
   MRCLASS = {28A80 (11K55 37B10 37B40)},
  MRNUMBER = {3636307},
       DOI = {10.1088/1361-6544/aa6078},
       URL = {https://doi.org/10.1088/1361-6544/aa6078},
}

@article {Feng-Hua-Ji-2007,
    AUTHOR = {Feng, De-Jun and Hua, Su and Ji, Yuan},
     TITLE = {When is the union of two unit intervals a self-similar set
              satisfying the open set condition?},
   JOURNAL = {Monatsh. Math.},
  FJOURNAL = {Monatshefte f\"ur Mathematik},
    VOLUME = {152},
      YEAR = {2007},
    NUMBER = {2},
     PAGES = {125--134},
      ISSN = {0026-9255,1436-5081},
   MRCLASS = {28A80 (11R06 28A75 37C70)},
  MRNUMBER = {2346429},
MRREVIEWER = {Bernd\ Sing},
       DOI = {10.1007/s00605-007-0499-8},
       URL = {https://doi.org/10.1007/s00605-007-0499-8},
}

@article {Wen-Zhao-2017,
    AUTHOR = {Wen, Zhi-Ying and Zhao, Xuan},
     TITLE = {A criterion for a finite union of intervals to be a
              self-similar set satisfying the open set condition},
   JOURNAL = {Asian J. Math.},
  FJOURNAL = {Asian Journal of Mathematics},
    VOLUME = {21},
      YEAR = {2017},
    NUMBER = {1},
     PAGES = {185--195},
      ISSN = {1093-6106,1945-0036},
   MRCLASS = {28A80 (28A75)},
  MRNUMBER = {3632440},
       DOI = {10.4310/AJM.2017.v21.n1.a6},
       URL = {https://doi.org/10.4310/AJM.2017.v21.n1.a6},
}

\end{document}